\documentclass[11pt,reqno]{amsart}
\usepackage{amssymb,amsmath,amsthm,verbatim,amstext,eucal,amscd}

\usepackage{fullpage}
\usepackage{mathtools}
\usepackage{verbatim}
\usepackage{enumitem}

\usepackage[all]{xy}

\numberwithin{equation}{section}

\newcommand{\id}{{\operatorname{id}}}
\newcommand{\ad}{\operatorname{ad}}
\newcommand{\proj}{\operatorname{proj}}
\newcommand{\ct}{Cartan triple}

\providecommand{\pauto}{\operatorname{pAut}}

\providecommand{\cliff}[1]{\operatorname{Cliff}(\pauto(#1))}
\providecommand{\fund}[1]{\operatorname{Fund}(\pauto(#1))}

\newtheorem{lemma}{Lemma}[section] 
\newtheorem{proposition}[lemma]{Proposition}
\newtheorem{theorem}[lemma]{Theorem}

\newtheorem{corollary}[lemma]{Corollary}

\newcommand{\FLEX}{\relax}
\newcommand{\flex}[1]{\renewcommand{\FLEX}{#1}}
\newtheorem{flexthm}[lemma]{\FLEX}

\theoremstyle{definition}
\newtheorem{definition}[lemma]{Definition}
\newtheorem{example}[lemma]{Example}

\newenvironment{remark}[1]{\refstepcounter{lemma}%
\vskip 5pt \par\noindent {\bf #1\ \thelemma .}}{\vskip 5pt \par}

\newenvironment{remark*}[1]{\par \vskip 5pt \noindent 
{\bf #1.}}{\vskip 5pt \par}

\newcommand{\Aut}{\operatorname{Aut}}
\newcommand{\aut}{\operatorname{Aut}}

\newcommand{\bh}{\ensuremath{{\mathcal B}({\mathcal H})}}

\newcommand{\cstaralg}{$C^*$-algebra}

\newcommand{\dstext}[1]{\quad\text{#1}\quad}
\newcommand{\eps}{\ensuremath{\varepsilon}}

\newcommand{\innerprod}[1]{\left\langle #1\right\rangle}

\newcommand{\mean}{\mathop{
     \mathchoice{\vcenter{\hbox{\huge
           $\Lambda$}}}{\Lambda}{\Lambda}{\Lambda}}
     \displaylimits}

\newcommand{\meet}{\mathop{
     \mathchoice{\vcenter{\hbox{\huge
           $\bigwedge$}}}{\wedge}{\wedge}{\wedge}}
     \displaylimits}

\newcommand{\norm}[1]{\left\|{#1}\right\|}

\newcommand{\ran}{\operatorname{range}}

\newcommand{\spn}{\operatorname{span}}

\newcommand{\vntensor}{\overline{\otimes}}

\newcommand{\bbC}{{\mathbb{C}}}

\newcommand{\bbN}{{\mathbb{N}}}

\newcommand{\bbR}{{\mathbb{R}}}

  \newcommand{\A}{{\mathcal{A}}}
  
  \newcommand{\B}{{\mathcal{B}}}
  \newcommand{\C}{{\mathcal{C}}}
  
  \newcommand{\D}{{\mathcal{D}}}
  
  \newcommand{\E}{{\mathcal{E}}}
  \newcommand{\F}{{\mathcal{F}}}
  \newcommand{\G}{{\mathcal{G}}}
   
\renewcommand{\H}{{\mathcal{H}}}

\renewcommand{\L}{{\mathcal{L}}}
  \newcommand{\M}{{\mathcal{M}}}
  \newcommand{\N}{{\mathcal{N}}}

\renewcommand{\P}{{\mathcal{P}}}
  
  \newcommand{\R}{{\mathcal{R}}}
  
\renewcommand{\S}{{\mathcal{S}}}
  \newcommand{\T}{{\mathcal{T}}}
  \newcommand{\U}{{\mathcal{U}}}

  \newcommand{\X}{{\mathcal{X}}}
  \newcommand{\Y}{{\mathcal{Y}}}
  \newcommand{\Z}{{\mathcal{Z}}}

\newcommand{\fA}{{\mathfrak{A}}}

\newcommand{\fH}{{\mathfrak{H}}}

\newcommand{\fJ}{{\mathfrak{J}}}

\newcommand{\fM}{{\mathfrak{M}}}

\newcommand{\fP}{{\mathfrak{P}}}

\newcommand{\pil}{\pi_\ell{}}
 
\newcommand{\fn}{\mathfrak{n}}
\newcommand{\fm}{\mathfrak{m}} 
\newcommand{\fp}{\mathfrak{p}}

\renewcommand{\>}{\rangle}

\newcommand{\ol}{\overline}

\newcommand{\atom}{\operatorname{atom}}
\newcommand{\sge}{\sb q}

\newcommand{\cocycle}{\sigma}

\begin{document}

\title[Cartan Triples]
{Cartan Triples} 
\author[A. P. Donsig]{Allan P. Donsig}\address{Dept. of
Mathematics\\ University of Nebraska-Lincoln\\ Lincoln, NE\\
68588-0130 } \email{adonsig@unl.edu}
\author[A. H. Fuller]{Adam H. Fuller} \address{Dept. of
Mathematics\\ Ohio University\\ Athens, OH\\
45701 } \email{fullera@ohio.edu} 
\author[D.R. Pitts]{David
R. Pitts} \thanks{This work was partially supported by a grant from the Simons
  Foundation (\#316952 to David Pitts).} \address{Dept. of Mathematics\\ University of
Nebraska-Lincoln\\ Lincoln, NE\\ 68588-0130}
\email{dpitts2@unl.edu}


\begin{abstract}
  We introduce the class of \ct s as a generalization of the notion of
  a Cartan MASA in a von Neumann algebra.  We obtain a one-to-one
  correspondence between \ct s and certain Clifford extensions of
  inverse semigroups.  Moreover, there is a spectral theorem
  describing bimodules in terms of their support sets in the
  fundamental inverse semigroup and, as a corollary, an extension of
  Aoi's theorem to this setting.  This context contains that of
  Fulman's generalization of Cartan MASAs and we discuss his
  generalization in an appendix.
\end{abstract}

\maketitle

\section{Introduction}\label{S: Intro}

As observed in the seminal work of
Feldman-Moore~\cite{FeldmanMooreErEqReI, FeldmanMooreErEqReII}, when a
von Neumann algebra contains a Cartan MASA, strong structural results
about the algebra may be obtained.  However, many von Neumann algebras
do not contain a Cartan MASA; the first examples were found
in~\cite{VoiculescuAnEnFiInMeFrPrThIII}.  Determining which von
Neumann algebras have a Cartan MASA and when it is unique is an
important question and has attracted significant attention; for two
examples, see~\cite{PopaVaesUnCaDeII1FaArFrArAcFrGr,
  PopaVaesUnCaDeII1FaArFrArAcHyGp}.  Part of the interest is that
Cartan MASAs are closely connected to crossed product decompositions,
as indeed is clear from the work of Feldman-Moore.

Recall a Cartan MASA $\D$ in a von Neumann algebra $\M$ is a maximal
abelian subalgebra with two additional properties: it is regular, that
is, the span of its normalizers is weak-$*$ dense in $\M$; and there
is a faithful, normal conditional expectation from $\M$ to $\D$.

In this paper, we study a much larger family of regular abelian von
Neumann subalgebras of von Neumann algebras.  Specifically, if $\M$ is
a von Neumann algebra, we consider an abelian and regular subalgebra
$\D\subseteq \M$ such that there is a faithful normal conditional
expectation onto the relative commutant $\D^c$ of $\D$ in
$\M$. Because $\D^c$ plays an important role in the structure of the
algebras, we name it $\N$ and call $(\M, \N, \D)$ a \textsl{Cartan
  triple}. 

In our previous work \cite{DonsigFullerPittsVNAlExInSe}, we showed
that Cartan MASAs can be described in terms of certain extensions of inverse
semigroups.  In the setting of Cartan triples, our main result is a
correspondence between Cartan triples and a larger class of extensions of inverse
semigroups
\[ \P\hookrightarrow \G\overset{q}\twoheadrightarrow \S. \]
Further, we obtain a Spectral Theorem for
$\N$-Bimodules and a version of Aoi's theorem in this context of
Cartan triples.  Although some of the methods
from~\cite{DonsigFullerPittsVNAlExInSe} extend naturally, significant
modifications are needed.  For example, for Cartan triples, $\P$ is not usually an abelian
inverse semigroup, but rather is Clifford, that is, the idempotents of
$\P$ commute with all elements of $\P$.

Various generalizations of MASAs have been considered in the
literature.  For example in~\cite{ExelNoCaSuC*Al}, Ruy Exel connects
the existence of a suitable non-abelian generalization of a MASA in a
\textit{separable} \cstaralg\ to a reduced crossed product
decomposition of the containing algebra.  Instead of the inverse
semigroup approach considered here, Exel considers a Fell bundle over
an inverse semigroup as the classifying structure.  The appropriate
variant of Exel's notion of a generalized Cartan subalgebra in the von
Neumann algebra setting is a full Cartan triple, meaning $\D$ is the center
of $\N$ (see
Definition~\ref{D: ct}).  Our approach using extensions of inverse
semigroups for classification, while related to Fell bundles, is
rather different.

Another generalization, also related to crossed product
decompositions, has been considered by Igor Fulman
in~\cite{FulmanCrPrvNAlEqReThSu}.  Fulman's generalization of a Cartan
MASA is, in our terms, a \ct\ with an additional condition, the
existence of a subgroup of the unitaries in $\M$ that normalizes $\D$,
contains the unitaries of $\N$, and has a suitable fixed point
property.  We show in Appendix A that this additional condition can be
characterized as the existence of a lift of an inverse semigroup
homomorphism from $\S$ into the partial automorphisms of the Cartan
triple.  Crossed products by inverse semigroups were first introduced
by N\'andor Sieben in~\cite{SiebenC*CrPrPaAcAcInSe} using such a
semigroup homomorphism into the partial automorphisms of a
C$^*$-algebra.  Thus, Fulman's condition can be interpreted naturally
as saying that the containing algebra is a crossed product by a
suitable inverse semigroup. Fulman's starting point was to generalize
the Feldman-Moore characterization of Cartan subalgebras
\cite{FeldmanMooreErEqReI, FeldmanMooreErEqReII} using measured
equivalence relations.  While some of our results resemble Fulman's,
ours are more general, perhaps because of 
the comparative simplicity of the inverse semigroup approach used
here.

We now discuss our results and their motivation in more detail.  We
associate to each Cartan triple an extension of inverse semigroups,
$\P \hookrightarrow \G \to \S$ where $\S$ is a fundamental inverse
semigroup, that is, the only elements commuting with the idempotents
of $\S$, denoted $\E(\S)$, is $\E(\S)$ itself, and $\P$ is Clifford,
meaning all elements of $\P$ commute with $\E(\S)$.  To be an
extension, the restriction to idempotents of the maps above must be
isomorphisms.  It is well known that every inverse semigroup $\G$ may
be represented as such an (idempotent-separating) extension; see
\cite[p.~141]{LawsonInSe}.

To construct the extension from a Cartan triple, take $\P$ to be the
partial isometries in $\N$ that normalize $\D$ and $\G$ to the partial
isometries in $\M$ that normalize $\D$, with $\P \hookrightarrow \G$ the inclusion
map.  To construct $\S$, we identify elements with the same action on
the idempotents, i.e., we quotient by the Munn congruence.

In~\cite{DonsigFullerPittsVNAlExInSe}, the inverse semigroup $\P$ was
abelian and we required that the character space of $\E(\P)$
was hyperstonean.  In that case, it was easy to recover $\D$ from
$\P$, as the continuous functions on the character space of $\E(\P)$.

Here, we need a condition that allows us to again recover $\D^c=\N$ from $\P$:  $\P$
arises as the partial isometries in a von Neumann algebra $\N$ which
normalize a fixed von Neumann subalgebra of the center of $\N$.  
In this case, we say that $\P$ is an $\N$-Clifford inverse monoid 
(Definition~\ref{def: N-clifford}).

To
see the need for this condition, consider the (degenerate) Cartan triple,
$(\M, \M, \bbC I)$.  In this case, the associated extension has the
form $\U(\M) \overset{\id}\to \U(\M) \to \bbC I$.  However, there are
von Neumann algebras not isomorphic to their opposite algebras
\cite{ConnesFaNoAnIsItWDetails}.  The unitary groups of such an
algebra and its opposite are isomorphic, so if the extension was
defined purely in terms of inverse semigroups (and without our
stronger condition) it would be possible for non-isomorphic triples to
produce the same extension.  

With these definitions in hand and the construction of an extension
from a triple (as outlined above), we show that Cartan triples are
isomorphic if and only if their extensions are (in a suitable sense)
isomorphic, Theorem~\ref{T: same data}.

To obtain the converse, we construct a Cartan triple from an extension
in Section~\ref{sec: rep G}. This is more subtle, and we build on the
strategy of our previous paper.  In particular, we use the order
structure of $\S$ to construct a reproducing kernel Hilbert
$\N$-bimodule, $\fA$.  We then define a representation, $\lambda$, of
$\G$, Theorem~\ref{T: Grepdef}, by partial isometries on $\fA$.  After
tensoring this representation with a faithful normal representation of
$\N$ to obtain a suitable representation of $\G$ (Theorem~\ref{T:
  Grepdef}), we define the Cartan triple of an extension in terms of
the double commutants of $\G$, $\P$, and their (common) idempotents
(Definition~\ref{ctripledef}).  In Theorem~\ref{T: indeppsi} we complete
the circle of ideas by showing that the extension associated to the
\ct\ constructed is (isomorphic to) the original extension.

In Section~\ref{sec: spectral thm} we begin a study of the
$\N$-bimodules in a \ct\ $(\M,\N,\D)$.  For our strongest results we
require that $\D$ be as large as possible, that is, $\D$ is the center
$\Z(\N)$ of $\N$.  We call
such a Cartan triple \textit{full}.  When $(\M,\D)$ is a
Cartan pair,  $(\M,\D,\D)$ is a full \ct, and so the class of full
\ct s properly includes Cartan pairs.  Different examples arise when
$\M$ is type I, Section~\ref{sec: examples type I}, and when $\M = \N
\rtimes_\alpha G$ is a crossed product of $\N$ by a discrete group $G$
which acts by properly outer automorphisms on $\N$ and $\Z(\N)$,
Theorem~\ref{thm: crossed prod}.

Let $(\M,\N,\D)$ be a full \ct, with associated extension $\P
\hookrightarrow \G \overset{q}\twoheadrightarrow \S$.  We show in
Theorem~\ref{plentyN} that if $B\subseteq \M$ is a non-zero weak-$*$
closed $\N$-bimodule, then $\G \cap B \neq \{0\}$.  Thus, every
weak-$*$ closed $\N$-bimodule gives rise to a non-trivial subset $q(B
\cap \G)$ of $\S$.  We call such sets \emph{spectral sets},
(Definition~\ref{def: spec set}).  If $A \subseteq \S$ is a spectral
set, then $\spn\{q^{-1}(A)\}$ is an $\N$-bimodule in $\M$.

Of course, it is conceivable that distinct weak-$*$ closed $\N$-bimodules
have the same spectral sets.  To study this, we 
use the Bures-topology on $\M$, induced by the conditional expectation
$E \colon \M \rightarrow \N$.  Whilst
the weak-$*$ and Bures topologies on $\M$ are not, in general,
comparable, the Bures-closed $\N$-bimodules are weak-$*$
closed~\cite[Lemma~3.1]{CameronSmithBiCrPrVNAl}.  The advantage of the
Bures topology over the weak-$*$ topology is that certain Fourier-type
series often converge in the Bures topology, while they need not
converge in the weak-$*$ topology (or in any other
``natural" topology).  Indeed, Mercer showed in
\cite{MercerCoFoDiCrPrVNAl} that the Fourier series of elements in
crossed-product von Neumann algebra converge in the Bures-topology,
but need not converge in the weak-$*$ topology.  Analogously, when $(\M,\N,\D)$ is a
\ct, we show in Theorem~\ref{BuresApx1} that if $x\in \M$, then $x$ is
the Bures-limit of the net of finite sums
\[\sum_{u\in F\subseteq \G\N(\M,\D)} u E(u^*x).\]
Similar results for $x$ in a Cartan pair are given in
\cite[Proposition~2.4.4]{CameronPittsZarikianBiCaMASAvNAlNoAlMeTh} and
\cite[Theorem~4.4]{MercerBiOvCaSu}.

In Proposition~\ref{samegn} we show that if $B$ is a weak-$*$ closed
$\N$-bimodule, $B_0 = \ol{\spn}^{wk^*}\{B \cap \G\}$, and $B_1 =
\ol{\spn}^{\mathrm{Bures}}\{B \cap \G\}$, then $B_0 \subseteq B
\subseteq B_1$ and each of $B_0$, $B$ and $B_1$ give the same spectral
sets.  We do not address when the bimodules $B_0$, $B$ and $B_1$ are
necessarily equal.  In
\cite{CameronPittsZarikianBiCaMASAvNAlNoAlMeTh}, if all weak-$*$
bimodules are necessarily Bures closed, the Cartan pair is said to
satisfy \emph{spectral synthesis}.  Even in the case of Cartan pairs,
whether $B_0=B_1$ remains an open problem.  There are some special cases when the
result is known.  If $(\M,\N,\D)$ is a \ct\ of the type studied by
Fulman \cite{FulmanCrPrvNAlEqReThSu} discussed above, with the added
condition that $\M$ is constructed from a hyperfinite equivalence
relation, then it can be deduced from Theorem~\ref{spthbi} and
\cite[Theorem~15.18]{FulmanCrPrvNAlEqReThSu} that all weak-$*$ closed
$\N$-bimodules are necessarily Bures closed.  Cameron and Smith
\cite{CameronSmithBiCrPrVNAl,CameronSmithInSuBiGeCrPrVNAl} studied a
related problem in crossed-products.  They showed that if $G$ is a
discrete group satisfying the AP condition, acting on a von Neumann
algebra $\N$ by properly outer automorphisms, then the weak-$*$ closed
$\N$-bimodules in $\N \rtimes_\alpha G$ are necessarily Bures closed.

We give a Spectral Theorem for Bimodules in Theorem~\ref{spthbi},
which gives a one-to-one correspondence between the Bures-closed
$\N$-bimodules and the spectral sets in $\S$.  
We find it striking
that Theorem~\ref{spthbi} depends
only on $\S$ and not on the extension $\G$.  Fuller and Pitts
\cite{FullerPittsIsLaBuClBiCaMA} had previously observed a similar
phenomenon: non-isomorphic Cartan pairs that have isomorphic lattices
of Bures-closed bimodules.
Theorem~\ref{spthbi} generalizes the Spectral Theorem for Cartan pairs found in
\cite{DonsigFullerPittsVNAlExInSe}; see also
\cite{CameronPittsZarikianBiCaMASAvNAlNoAlMeTh}.  It should be noted
that the study of bimodules in Cartan pairs was initiated in the
seminal work of Muhly, Saito and Solel \cite{MuhlySaitoSolelCoTrOpAl}.
They present a spectral theorem for weak-$*$ closed bimodules.  Their
work, however, has a gap.  Though not explicitly stated as such, the
gap in \cite{MuhlySaitoSolelCoTrOpAl} amounts to assuming that the
weak-$*$ closed bimodules are necessarily Bures closed, see
\cite{CameronPittsZarikianBiCaMASAvNAlNoAlMeTh}.  

A class of $\N$-bimodules of particular interest are the von Neumann
algebras $\L$ such that $\N \subseteq \L \subseteq \M$.  In
Theorem~\ref{Aoi} we show that if $(\M,\N,\D)$ is a full \ct\ and
$\N \subseteq \L \subseteq \M$, then $(\L,\N,\D)$ is again a \ct.
This extends Aoi's result for intermediate von Neumann algebras in
Cartan pairs \cite{AoiCoEqSuInSu}.  A key step in the proof is showing
that an intermediate subalgebra $\L$ is necessarily closed in the
Bures topology.  Thus, Theorem~\ref{Aoi} together with
Theorem~\ref{spthbi} immediately give a one-to-one correspondence
between the intermediate von Neumann subalgebras containing $\N$, and
the  sub-inverse Cartan monoids of $\S$, Corollary~\ref{cor:
  galois}.  We view this as a Galois correspondence-type result;
although we do not have a group to hand, there is the Cartan inverse
semigroup.  Corollary~\ref{cor: galois} should be compared with the
following well-known result: If $\N$ is a factor and $G$ is a discrete
group acting on $\N$ by (properly) outer automorphisms, then Izumi,
Longo and Popa \cite{IzumiLongoPopaGaCoCoGrAuVNAlGeKaAl} show that if
$\L$ is a von Neumann algebra satisfying
$\N \subseteq \L \subseteq \N \rtimes_\alpha G$ then there is a
subgroup $H$ of $G$ such that $\L = \N \rtimes_\alpha H$; see also
\cite{CameronSmithBiCrPrVNAl, ChodaGaCoVNAl}.  That is, there is a
one-to-one correspondence between subgroups of $G$ and the von Neumann
algebras $\M$ with $\N \subseteq \M \subseteq \N \rtimes_\alpha G$.  A
similar Galois correspondence-type theorem without an explicit group
structure has been obtained
in~\cite{AoiYamanouchiOnNoGrCoGrInFaErEqReSu} for von Neumann algebras
generated by a measured equivalence relation and an appropriate
cocycle.

Cameron and Smith have considered similar questions in
\cite{CameronSmithBiCrPrVNAl,
  CameronSmithGaCoReCrPrUnSiC*AlDiGp,CameronSmithInSuBiGeCrPrVNAl}.
They study crossed products by discrete groups and the bimodule and
intermediate algebra structure therein, amongst other things.  There
is overlap with our work and \cite{CameronSmithInSuBiGeCrPrVNAl}, with
neither work subsuming the other.  There they let $\N$ be any von
Neumann algebra and let $G$ be a discrete group acting on $\N$ by
properly outer automorphisms.  If $\N$ is abelian, then $\N$ is a
Cartan MASA in $\N\rtimes_\alpha G$ and so both our settings cover
this case.  If $\N$ is not abelian, but $G$ also acts on $\Z(\N)$ by
properly outer automorphisms then it is shown in Theorem~\ref{thm:
  crossed prod} $(\N\rtimes_\alpha G, \N, \Z(\N))$ is a \ct.

\section{Cartan triples and their extensions} \label{S: prelim}

Our main goals in this section are the construction of the extension
associated to a Cartan triple, Proposition~\ref{propextension}, and
the result that two such extensions are isomorphic if and only if they arise from
isomorphic Cartan triples, Theorem~\ref{T: same data}.

We begin by fixing some notation.  For $\M$ a von Neumann algebra,
$\Z(\M)$ denotes its center and $\U(\M)$ the unitary elements.  For
$\X \subseteq \M$, $\X^c$ denotes the relative commutant, that is,
\[\X^c:=\{m\in \M: xm=mx \text{ for all } x\in \X\}.\]

\begin{definition} Suppose $\M$ and $\L$ are von Neumann algebras with
$\L\subseteq \M$.  The \textit{groupoid normalizer of $\L$ in $\M$} is
the set,
\[\G\N(\M,\L):=\{v\in \M: v \text{ is a partial isometry and $v^*\L
v\cup v\L v^*\subseteq \L$}\}.\] If the linear span of $\G\N(\M,\L)$
is weak-$*$ dense in $\M$, we say $\L$ is a \textit{regular subalgebra
of $\M$} or that $\L$ is \textit{regular} in $\M$.
\end{definition}

\begin{remark}{Remark} When $\L\subseteq \M$ is an abelian von Neumann
subalgebra of $\M$, it is more common to say $\L$ is regular in $\M$
if $\spn\{U\in\U(\M): U\D U^*=\D\}$ is weak-$*$ dense.  However, if
$\L$ is abelian, then
\begin{equation}\label{ungn} \spn\{U\in\U(\M): U\D
U^*=\D\}=\spn\G\N(\M,\D),
\end{equation} thus the two definitions coincide in this case.  For a
proof of this statement see \cite[p.~479,
Inclusion~2.8]{CameronPittsZarikianBiCaMASAvNAlNoAlMeTh}.
\end{remark}

We now introduce our main topic of study.

\begin{definition}\label{D: ct}  A \textit{Cartan triple} is a triple $(\M,\N,\D)$
consisting of three von Neumann algebras satisfying:
\begin{enumerate}
\item $\D$ is an abelian and regular von Neumann subalgebra of $\M$;
\item $\N$ is the relative commutant of $\D$ in $\M$; and
\item there exists a faithful normal conditional expectation
$E:\M\rightarrow \N$.
\end{enumerate} A Cartan triple $(\M,\N,\D)$ is \textit{full} when
$\D=\Z(\N)$.
\end{definition}

\begin{remark}{Remarks}\label{fullerpair}\begin{enumerate}
\item For any \ct\ $(\M,\N,\D)$, $\M\supseteq
\N\supseteq \D$ because $\D$ is abelian.
\item If  $(\M,\N,\D)$ is a \ct, then $(\M,\N,\Z(\N))$ is a full \ct.  Indeed,
$\N=\D^c$ implies  $\N=(\Z(\N))^c$, and since
$\G\N(\M,\D)\subseteq \G\N(\M,\Z(N))$, $\Z(\N)$ is regular
in $\M$.
\end{enumerate}
\end{remark}

Section~\ref{sec: examples} is devoted to examples of \ct s.  Here we
content ourselves with making two simple observations regarding what
occurs when two of the von Neumann algebras in a Cartan triple
coincide.

\begin{remark}{Examples}\label{2algagree}
\begin{enumerate}
\item Suppose $(\M,\D,\D)$ is a Cartan triple.  Then $(\M,\D,\D)$ is
full and $\D$ is a Cartan MASA in $\M$.  In this sense, the class of
Cartan triples includes the class of Cartan pairs.
\item Now suppose $(\N,\N,\D)$ is a Cartan triple.  Then $\D^c=\N$.
When $\N$ has separable predual, we may write $\D=L^\infty(X,\mu)$ and
write $\N=\int^{\oplus}_X \N_x\, d\mu(x)$ as a direct integral.  When
this is done, $\G\N(\N,\D)$ may be identified with the set of all
functions $f\in \int_X^\oplus \N_x\, d\mu(x)$ such that for almost
every $x\in X$, $f(x)\in\U(\N_x)\cup\{0\}$.
\end{enumerate}
\end{remark}

Example~\ref{2algagree}(b) shows how the inverse semigroup
$\G\N(\N,\D)$ can be used to describe a direct integral.  Further,
this inverse semigroup approach allows one to work with von Neumann
algebras which do not have separable predual.  This example discussed
further in Example~\ref{ex: direct int 2}.

We fix
some notation for inverse semigroups next.  For the most part, our
notation follows Section 2 of
\cite{DonsigFullerPittsVNAlExInSe}, which also gives much of the inverse semigroup
theory we will use.  For an in-depth text on inverse semigroups see
\cite{LawsonInSe}.  Throughout the paper:
\begin{itemize}
 \item $\E(\S)$ will denote the idempotents of the inverse
semigroup $\S$;
\item we use $s^\dag$ to denote the inverse of the element $s$ in an abstract inverse
semigroup; however, for an inverse semigroup of partial isometries on
a Hilbert space, the adjoint $v^*$ is the inverse of the element $v$
and we typically use $v^*$ instead of $v^\dag$ in this setting.
\end{itemize}

For our extensions, we need two special classes of inverse monoids:
Cartan inverse monoids, defined in
\cite{DonsigFullerPittsVNAlExInSe}, and $\N$-Clifford
inverse monoids, which are new.

\begin{definition}[{\cite[Definition~2.11]{DonsigFullerPittsVNAlExInSe}}]
\label{D: Cartan inv}
  We call an inverse semigroup $\S$ a \emph{Cartan inverse monoid} if
	\begin{enumerate}
	\item $\S$ is fundamental;
	\item $\S$ is a complete Boolean inverse monoid; and
	\item the character space $\widehat{\E(\S)}$ of the complete
Boolean lattice $\E(\S)$ is a hyperstonean topological space.
	\end{enumerate}
\end{definition}

\begin{definition}\label{def: N-clifford} Let $\N$ be a von Neumann
algebra.  An \textit{$\N$-Clifford inverse monoid} is an inverse
monoid $\P$ such that $\P=\G\N(\N,\D)$, where $\D$ is a von Neumann
subalgebra of $\Z(\N)$.  If in addition $\D=\Z(\N)$, we say $\P$ is a
\textit{full $\N$-Clifford inverse monoid}.
\end{definition}

\begin{remark}{Remark} Suppose $\P=\G\N(\N,\D)$ is an $\N$-Clifford
inverse monoid.  It is not difficult to show that
\[\P=\{v\in\N: v \text{ is a partial isometry with }
vv^*=v^*v\in\D\}\dstext{and} \E(\P)=\proj(\D).\] In particular, $\P$
is a Clifford inverse semigroup of partial isometries.

Since $\U(\N)\subseteq \P$, every element of $\N$ is a linear
combination of at most four elements of $\P$.
\end{remark}

We need an appropriate notion of isomorphism of such inverse monoids.
\begin{definition} If for $i=1,2$, $\P_i$ are $\N_i$-Clifford inverse
monoids, a map $\alpha:\P_1\rightarrow\P_2$ is an \textit{extendible
isomorphism} if there exists a normal $*$-isomorphism
$\theta:\N_1\rightarrow\N_2$ such that $\alpha=\theta|_{\P_1}$;
equivalently, there exists a normal isomorphism
$\theta:\N_1\rightarrow \N_2$ such that $\theta(\D_1)=\D_2$.
\end{definition} Obviously, any extendible isomorphism is an
isomorphism of inverse semigroups.

\begin{definition} For $i=1,2$, let $\S_i$ be Cartan inverse monoids,
let $\P_i$ be $\N_i$-Clifford inverse monoids, and suppose
$\P_i\xhookrightarrow{\iota_i} \G_i \overset{q}\twoheadrightarrow \S_i$ are
extensions of $\S_i$ by $\P_i$.  These extensions are
\textit{equivalent} if there are semigroup isomorphisms
$\alpha:\G_1\rightarrow \G_2$, $\tilde{\alpha}:\S_1\rightarrow \S_2$
and an extendible isomorphism $\underline \alpha: \P_1\rightarrow
\P_2$ such that
\[\alpha\circ \iota_1=\iota_2\circ\underline\alpha\dstext{and}
q_2\circ \alpha= \tilde\alpha \circ q_1.\]
\end{definition}

\begin{remark}{Remark} When $\P_i$ is the set of partial isometries in
$C(\widehat{\E(\S_i)})$ and $\N_i:=C(\widehat{\E(\S_i)})$, then this
definition reduces to the notion of equivalence for extensions found
in~\cite{DonsigFullerPittsVNAlExInSe}.
\end{remark}

\begin{definition}\label{Munn} Let $\G$ be an inverse semigroup.  The
\textit{Munn congruence} (also called the \textit{Munn relation}) on
$\G$ is the set
\[R_M:=\{(v_1,v_2)\in \G\times \G: v_1ev_1^\dag=v_2ev_2^\dag \text{
for all } e\in \E(\G)\}.\] Then $R_M$ is the maximal idempotent
separating congruence on $\G$ and the set of $R_M$-equivalence classes
equipped with the product $[v][w]=[vw]$ and inverse
$[v]^\dag=[v^\dag]$ form a fundamental inverse
semigroup~\cite[Proposition~5.2.5]{LawsonInSe}.
\end{definition}

We now show how a \ct\ gives rise to an extension of inverse
semigroups.

\begin{proposition}\label{propextension} Let $(\M,\N,\D)$ be a \ct\
and set
\[\G:=\G\N(\M,\D)\dstext{and} \P:=\G\N(\N,\D).\] Then $\G$ and $\P$
are inverse semigroups with $\P\subseteq \G$ and
\[\E(\P)=\E(\G)=\textit{proj}(\D).\] Moreover, the following
statements hold.
\begin{enumerate}
\item $\P$ is a $\N$-Clifford inverse monoid. 
\item $\P$ is the set of elements of $\G$ Munn-related to an
idempotent.
\item If $\S$ is the quotient of $\G$ by the Munn congruence, then
$\S$ is a Cartan inverse monoid.
\end{enumerate}
\end{proposition}

\begin{proof} Obviously, $\P\subseteq \G$ and by definition, $\P$ is
  an $\N$-Clifford inverse monoid.  If $v\in \G$, then
$v^*v\in\D$, so every idempotent of $\G$ is a projection in $\D$.
Also,
every projection in $\D$ is an idempotent in $\G$.  It follows that
$\G$ and $\P$ are von Neumann regular monoids for which the
idempotents commute, so both are inverse monoids and
\[\E(\P)=\E(\G)=\text{proj}(\D).\]

If $v\in\P$, then $(v,vv^*)\in R_M$, so every element of $\P$ is
Munn-related to an idempotent.  On the other hand, suppose $v\in
\G\N(\M,\D)$ is Munn-related to the idempotent $e$.  Then $vv^*=e$.
Let $p$ be a projection in $\D$.  Then
\[vp=vpv^*v=epv=pev=pv.\] Hence $v\in \N$, and so $v\in\P$.   Thus,
$\P$ is the set of elements of $\G$ Munn related to an idempotent.

We have already observed that $\S$ is a fundamental inverse monoid,
and it clearly contains a zero element $0$.  As the Munn congruence is
idempotent separating, $\E(\S)$ is isomorphic to $\E(\P)$.  The proof
that $\S$ is a Cartan inverse monoid now follows exactly as in the
proof of~\cite[Proposition~3.5]{DonsigFullerPittsVNAlExInSe}.
\end{proof}

As noted in the introduction, any inverse semigroup $G$ may be represented
as an extension of a fundamental inverse semigroup $S$  by a Clifford
inverse semigroup $C$.  Indeed, $C$ may be taken to be the set
of elements of $G$ which are Munn-related to an idempotent, and $S$ is
the quotient of $G$ by the Munn relation.    We apply this
construction to the inverse semigroup $\G$ of
Proposition~\ref{propextension} to obtain the class of extensions
studied in this paper.

\begin{definition} Let $(\M,\N,\D)$ be a \ct.  Put $\G:=\G\N(\M,\D)$,
$\P:=\G\N(\N,\D)$, $\S:=\G/R_M$, and let $q:\G\rightarrow \S$ the
quotient map.  The extension
\begin{equation}\label{ext} \P\hookrightarrow \G\overset{q}\twoheadrightarrow \S.
\end{equation} is called the \textit{extension associated to
$(\M,\N,\D)$}.
\end{definition}

\begin{example}\label{ex: direct int 2} We return to the context of
  Example~\ref{2algagree}(b), that is, of a Cartan triple having the
  form $(\N,\N,\D)$.  In this setting, $\P=\G$ consists of the partial
  isometries in $\N$ whose initial and final spaces coincide and
  belong to $\D$; $\S$ is the projection lattice of $\D$; and $q$ is
  the map $v\mapsto v^*v$.  Note that $\N$ is the linear span of $\P$.
  When $\N_*$ is separable and $\N$ is identified as the direct
  integral $\int^\oplus_X \N_x\, d\mu(x)$, we may view the extension
  $\P\hookrightarrow \P\overset{q}\twoheadrightarrow\S$ as giving a description of
  the direct integral in terms of the linear span of
  \[\left\{ f\in\int^\oplus_X \M_x\, d\mu(x): f(x)\in \U(\M_x)\cup
      \{0\} \text{ for almost every $x$}\right\}.\] The extension
  approach encodes the measure theory into $\D$, and is a more
  operator theoretic view of $\N$ as opposed to the point based view
  of $\N$ as a direct integral.
\end{example}

In the study of extensions, it is often useful to choose a section $j$
for the quotient map $q$, that is, $j$ is a map such that
$q\circ j=\text{id}|_\S$.  In our context, we will frequently need a
section which is \textit{order preserving} in the sense that $j(1)=1$
and whenever $s, t\in\S$ and $s\leq t$, we have $j(s)\leq j(t)$
(see~\cite[Definition~4.1]{DonsigFullerPittsVNAlExInSe}).  Most of the
following result was proved in \cite{DonsigFullerPittsVNAlExInSe} when
$\P$ is the set of partial isometries in $C^*(\E(\S))$, but the same
proof holds for extensions of $\S$ by $\N$-Clifford inverse monoids
considered here.  

Recall that $q|_{\E(\G)}$ is a complete Boolean
algebra isomorphism of $\E(\G)$ onto $\E(\S)$.
Also, as observed in~\cite[Remark~4.8]{DonsigFullerPittsVNAlExInSe},
for any $s, t\in\S$, $(s^\dag t\wedge 1)$ is the source idempotent for
$s\wedge t$, that is, 
\begin{equation}\label{sourceidem}
(s\wedge t)^\dag (s\wedge t)=s^\dag t\wedge 1.
\end{equation}

\begin{proposition}[c.f.~{\cite[Proposition~4.6]{DonsigFullerPittsVNAlExInSe}}]\label{P: order per sec} Let $\P\hookrightarrow
  \G\overset{q}\twoheadrightarrow \S$ be an extension of the Cartan
  inverse monoid $\S$ by the $\N$-Clifford inverse monoid $\P$.  The
  map $(q|_{\E(\S)})^{-1}$ extends to an order preserving section
  $j:\S\rightarrow \G$ for $q$.  This section has the property that
  for $s_1, s_2\in\S$,
\begin{equation}\label{ops1} j(s_1)^\dag j(s_2) j(s_1^\dag s_2\wedge
1)=j(s_1^\dag s_2\wedge 1).
\end{equation}
\end{proposition}

\begin{proof} Equation~\eqref{ops1} was not proved in
\cite{DonsigFullerPittsVNAlExInSe}, so we provide a proof here.  
Using~\eqref{sourceidem}, observe
\begin{align*} j(s_1)^\dag j(s_2) j(s_1^\dag s_2\wedge 1) &=
j(s_1)^\dag j(s_1\wedge s_2)\\ &= \left(j(s_1\wedge s_2) j(s_1\wedge
s_2)^\dag j(s_1)\right)^\dag j(s_1\wedge s_2)\\ &=j(s_1\wedge
s_2)^\dag j(s_1\wedge s_2)=j(s_1^\dag s_2\wedge 1). \qedhere
\end{align*}
\end{proof}

\begin{definition} We say that the \ct s $(\M_1,\N_1,\D_1)$ and
$(\M_2,\N_2,\D_2)$ are \textit{isomorphic} if there exists a normal
$*$-isomorphism $\theta:\M_1\rightarrow\M_2$ such that
$\theta(\D_1)=\D_2$.
\end{definition}

Our goal is to show that \ct s, up to isomorphism, are uniquely
determined by their associated extensions, up to equivalence.  We do
this in Theorem~\ref{T: same data}.  We first need some technical
lemmas.  

Throughout the remainder of the section, fix an order-preserving
section $j$ for the extension
$\P\hookrightarrow\G\overset{q}\twoheadrightarrow \S$ associated to the \ct\ $(\M,\N,\D)$.

\begin{lemma}\label{Emap} Let $(\M,\N,\D)$ be a \ct\ with conditional
expectation $E:\M\rightarrow \N$.  Let $\P\hookrightarrow \G\overset{q}\twoheadrightarrow \S$ be the associated extension.  Then $E(\G)=\P$.
Further, for $v\in\G\N(\M,\D)$,
\[E(v)=ve,\dstext{where} e=j(q(v)\wedge 1)\in\E(\G).\] Also, $E$ preserves the natural inverse semigroup partial order
on $\G$ and $\P$ in the sense that if $v,w\in\G$ with $v\leq w$, then
$E(v)\leq E(w)$.
\end{lemma}

\begin{proof} Each $v\in \G$ induces a normal $*$-isomorphism $\theta_v$ from
$v^*v\D$ to $vv^*\D$, given by $\theta_v(d) = vdv^*$.  Applying Frol\'{i}k's
Theorem~\cite[Proposition~2.11A]{PittsStReInI} to $\theta_v$ we may find
elements $e_0, e_1, e_2, e_3\in\E(\G)$ such that
\begin{enumerate}
\item for $i \neq j$, $e_i\wedge e_j =0$;
\item $e_0\vee e_1 \vee e_2\vee e_3=v^*v$;
\item for $i=1,2,3$, $(ve_i)^2=0$; and
\item $q(ve_0)\in \E(\S).$
\end{enumerate} If $w\in\G$ and $w^2=0$, then
\[E(w)=ww^*E(w)=E(w)ww^*=E(www^*)=0.\] It follows that
\[E(v)=E(ve_0)=ve_0.\]
Let $e:=e_0$.  Then $q(e)=q(v)\wedge 1$
because $e_0$ corresponds to the ideal of $\D$ consisting of all
elements fixed by $\theta_v$.

Finally, if $v,w\in \G$ and $v\leq w$, we may find $f\in\E(\G)$ so
that $v=wf$.  Then $E(v)=E(w)f$, so $E(v)\leq E(w)$.
\end{proof}

\begin{lemma}\label{righton} Let $(\M,\N,\D)$ be a \ct, and suppose
$y$ belongs to the linear span of $\G\N(\M,\D)$.  Then there exists a
finite set $\{w_k\}_{k=1}^m\subseteq \G\N(\M,\D)$ such that
$E(w_j^*w_k)=0$ for $j\neq k$ and
\begin{equation}\label{righton1} y=\sum_{k=1}^m w_kE(w_k^*y).
\end{equation}
\end{lemma}
\begin{proof} Choose $\{v_j\}_{j=1}^N\subseteq \G\N(\M,\D)$ and
scalars $\{c_j\}_{j=1}^N$ so that $c_jv_j\neq 0$ for each $j$ and
$y=\sum_{j=1}^N c_jv_j$.  Let $s_i:=q(v_i)$ and
apply~\cite[Lemma~4.15]{DonsigFullerPittsVNAlExInSe} to obtain a
finite set $\{t_k\}_{k=1}^m\subseteq \S$ satisfying
\begin{enumerate}
\item for $1\leq j\leq m$, $t_j\neq 0$;
\item for $j\neq k$, $t_j\wedge t_k =0$;
\item for $1\leq j\leq m$ and $1\leq n \leq N$, $t_j\wedge s_n \in
\{0,t_j\}$;
\item for $1\leq j\leq m$ there exists $1\leq n \leq N$ such that
$t_j\wedge s_n=t_j$; and
\item for each $1\leq n \leq N$, $s_n = \bigvee\{t_j \colon t_j\leq
s_n\}$.
\end{enumerate}

Let $w_k:=j(t_k)$.  Lemma~\ref{Emap} implies that $E(w_j^*w_k)=0$ when
$j\neq k$.

For $1\leq n\leq N$, let $I_n:=\{i: t_i\leq s_n\}$.  Given $n$ and
$i\in I_n$, another application of Lemma~\ref{Emap} gives
\[w_iE(w_i^*v_n)=w_iw_i^*v_nj(t_i^\dag s_n\wedge 1)=v_n
(v_n^*w_iw_i^*v_n)j(t_i^\dag s_n\wedge 1)=v_nj(t_j^\dag s_n\wedge
1).\] Since $s_n=\bigvee_{i\in I_n} (t_i\wedge s_n)$, we obtain,
\[v_n=\sum_{i=1}^m w_iE(w_i^*v_n).\] Equation~\eqref{righton1} follows.
\end{proof}

We now recall notation regarding weights on a von Neumann algebra used
in~\cite{TakesakiThOpAlII}.
Suppose $\M$ is a von Neumann algebra and $\phi$ is a weight on $\M$.
Recall that
\begin{align*}
\fp_\phi&:=\{x\in \M_+: \phi(x)<\infty\}, \quad
 \fn_\phi:=\{x\in\M:
\phi(x^*x)<\infty\}\text{ and}\\ \fm_\phi&:=\left\{\sum_{k=1}^N y_k^*x_k:
n\in \bbN,  x_k, y_k \in \fn_\phi\right\}.
\end{align*}
By~\cite[Lemma~VII.1.2]{TakesakiThOpAlII}, $\fp_\phi$ is a hereditary
convex cone in $\M_+$, $\fn_\phi$ is a left ideal of $\M$, 
$\fm_\phi$ is a hereditary $*$-algebra of $\M$, and every element of
$\fm_\phi$ is a linear combination of four elements of $\fp_\phi$.
The semi-cyclic representation $\pi_\phi$ of $\M$ on the Hilbert space
$\H_\phi$ associated to $\phi$ will be
denoted $(\pi_\phi,\H_\phi, \eta_\phi)$.  When $\phi$ is faithful,
normal and semi-finite, $\pi_\phi$ is a faithful, normal representation
of $\M$.

\begin{lemma}\label{semifinite} Suppose $\M$ is a von Neumann algebra,
$\L\subseteq \M$ is a von Neumann subalgebra, and there exists a
faithful normal conditional expectation $E:\M\rightarrow \L$.  Let
$\psi$ be a faithful, normal, semi-finite weight on $\L$ and let
$\phi:=\psi\circ E$.  Then $\phi$ is a faithful normal semi-finite
weight on $\M$.
\end{lemma}

\begin{proof} 
Since $\fm_\psi$ is a $*$-subalgebra of $\L$, a
corollary of the Kaplansky density theorem shows there exists a net
$(x_\lambda)$ in $\fp_\psi$ with $0\leq x_\lambda \leq I$ which
converges $\sigma$-strongly to $I$.  For any $z\in \M_+$, $x_\lambda
zx_\lambda\leq \norm{z} p_\lambda^2\in\fp_\psi$.  Therefore,
$x_\lambda zx_\lambda\in \fp_\phi$.  Since
$\lim^{\sigma\text{-strong}} x_\lambda zx_\lambda =z$, we have that
$\fp_\phi$ generates $\M$.  That is, $\phi$ is semi-finite on $\M$.
\end{proof}

The following result is the key technical tool used in the proof of
Theorem~\ref{T: same data}.

\begin{lemma}\label{fndense} Let $(\M,\N,\D)$ be a \ct, suppose $\psi$
is a faithful normal semi-finite weight on $\N$ and let
$\phi=\psi\circ E$.  Then the linear span of $\{\eta_\phi(vn):
v\in\G\N(\M,\D) \text{ and } n\in\fn_\psi\}$ is dense in $\H_\phi$.
\end{lemma}

\begin{proof} For $m\in\fn_\phi$ and $v\in \G\N(\M,\D)$,
\[\psi(E(m^*v)v^*vE(v^*m))=\psi(E(m^*v)E(v^*m))\leq
\psi(E(m^*vv^*m))\leq \psi(E(m^*m))=\phi(m^*m)<\infty.\] This yields
the following.
\begin{enumerate}
\item $E(\fn_\phi)=\fn_\psi$ (take $v=I$).
\item For $v\in\G\N(\M,\D)$, the map $\M\ni m\mapsto vE(v^*m)$ is
idempotent and leaves $\fn_\phi$ invariant; hence there is a
projection $P_v\in\B(\H_\phi)$ whose action on $\eta_\phi(\fn_\phi)$
is given by $P_v\eta_\phi(m)=\eta_\phi(vE(v^*m))$.  In addition,
notice that $\ran P_v=\overline {\{\eta_\phi(vn): n\in \fn_\psi\}}$.
\end{enumerate} To prove the lemma, it therefore suffices to show that
if $\xi\in\H_\phi$ and $P_v\xi=0$ for every $v\in\G\N(\M,\D)$, then
$\xi=0$.  We begin with a preliminary fact
about approximating norms of vectors in $\H_\phi$.

Let
\[\Phi:=\{\tau\circ E: \tau\in\N_*^+\text{ and } \tau(n)\leq
\psi(n)\text{ for all } 0\leq n\in \N\}.\] Clearly $\Phi\subseteq
\M_*^+$.  For $\omega\in\Phi$, let $(\pi_\omega,\H_\omega,
\eta_\omega)$ be the semi-cyclic representation of $\M$ arising from
$\omega$.  This representation is actually cyclic and $\fn_\omega=\M$
because $\omega$ is a bounded positive linear functional on $\M$.
Define $T_\omega: \eta_\phi(\fn_\phi)\rightarrow \H_\omega$ by
$T_\omega\eta_\phi(m)=\eta_\omega(m)$.  Write $\omega=\rho\circ E$ for
some $\rho\in\N_*^+$.   Then for $m\in\fn_\phi$,
\[\norm{\eta_\omega(m)}^2=\rho(E(m^*m))\leq\psi(E(m^*m))=\norm{\eta_\phi(m)}^2.\]
Thus $T_\omega$ extends to a contraction belonging to $\B(\H_\phi,
\H_\omega)$, which we again denote by $T_\omega$. 

 We claim that for
any $\xi\in\H_\phi$,
\begin{equation}\label{approxi} \norm{\xi}=\sup_{\omega\in\Phi}
\norm{T_\omega\xi}.
\end{equation} To see this, fix $\xi\in\H_\phi$ and choose a real
number $r$ such that $r<\norm{\xi}$.  Let $\eps>0$ satisfy $3\eps <
\norm{\xi}-r$.  Choose $m\in \fn_\phi$ such that
$\norm{\xi-\eta_\phi(m)}<\eps$.  By Haagerup's Theorem
(see~\cite[Theorem~VII.1.11]{TakesakiThOpAlII}),
\[\psi(E(m^*m))=\sup\{\tau(E(m^*m)): \tau\in \N_*^+ \text{ and
}\tau(n)\leq \psi(n) \text{ for all } 0\leq n\in\N\}.\] Hence there
exists $\omega\in\Phi$ such that
\[\norm{T_\omega\eta_\phi(m)}> \norm{\eta_\phi(m)}-\eps.\] Then
\begin{align*} \norm{\xi}&\leq
\norm{\xi-\eta_\phi(m)}+\norm{\eta_\phi(m)} <2\eps
+\norm{\eta_\phi(m)}-\eps\\ &< 2\eps+\norm{T_\omega\eta_\phi(m)}\\
&\leq 3\eps +\norm{T_\omega\xi} < \norm{\xi}-r +\norm{T_\omega\xi},
\end{align*} whence $r<\norm{T_{\omega}\xi}$.  Thus~\eqref{approxi}
holds.

For each $\omega\in \Phi$ and $v\in\G\N(\M,\D)$, let $P_v^\omega$ be
the projection on $\H_\omega$ determined by $\eta_\omega(m)\mapsto
\eta_\omega(vE(v^*m))$, $m\in\M$.  A routine calculation shows that
for every $\omega\in \Phi$, $m\in \fn_\phi$ and $v\in\G\N(\M,\D)$,
$T_\omega P_v\eta_\phi(m)=P_v^\omega T_\omega\eta_\phi(m)$, so
\begin{equation}\label{Pvintertwine} T_\omega P_v=P_v^\omega T_\omega.
\end{equation}

Fix $\omega\in \Phi$.  We claim that if $\zeta\in\H_\omega$ and
$P_v^\omega\zeta=0$ for every $v\in \G\N(\M,\D)$, then $\zeta=0$.
Suppose $\zeta\in \H_\omega$ is such a vector.  Given $\eps>0$, there
exists $x\in \M$ such that $\norm{\zeta-\eta_\omega(x)}<\eps$.  Since
$\D$ is regular in $\M$, and $\omega\in \M_*^+$, there exists
$y\in\spn\G\N(\M,\D)$ such that $\norm{\eta_\omega(x-y)}<\eps$.  By
Lemma~\ref{righton}, there exists a finite $E$-orthogonal set
$\{v_k\}_{k=1}^m\subseteq \G\N(\M,\D)$ such that $y=\sum_{k=1}^m
v_kE(v_k^*y)$.  As $\{P_{v_k}^\omega\}_{k=1}^n$ is a pairwise
orthogonal set of projections, $Q:=\sum_{k=1}^m P_{v_k}^\omega$ is a
projection.  Therefore,
\[\eta_\omega(y)=\sum_{k=1}^m P_{v_k}^\omega
\eta_\omega(y)=Q\eta_\omega(y).\] So,
\begin{align*} \norm{\zeta}&\leq
\norm{\zeta-\eta_\omega(x)}+\norm{\eta_\omega(x-y)}+\norm{\eta_\omega(y)}\\
&< 2\eps +\norm{\eta_\omega(y)-Q\zeta}=2\eps
+\norm{Q(\eta_\omega(y)-\zeta)}\\ &\leq 2\eps
+\norm{\eta_\omega(y)-\zeta}\\ &<4\eps.
\end{align*} Thus the claim holds.

Now suppose $\xi\in\H_\phi$ satisfies $P_v\xi=0$ for every
$v\in\G\N(\M,\D)$.  Then for every $\omega\in\Phi$ and
$v\in\G\N(\M,\D)$,
\[0=T_\omega P_v\xi=P_v^\omega T_\omega\xi.\] Hence $T_\omega\xi=0$
for every $\omega\in\Phi$, so $\xi=0$ by~\eqref{approxi}.  The proof
is now complete.
\end{proof}

We come now to the main theorem of this section.

\begin{theorem} \label{T: same data} 
The Cartan triples $(\M_1,\N_1,\D_1)$ and
$(\M_2,\N_2,\D_2)$ are isomorphic if and only if their associated
extensions, $\P_1\hookrightarrow\G_2\overset{q_1}\twoheadrightarrow \S_2$ and
 $\P_1\hookrightarrow\G_2\overset{q_2}\twoheadrightarrow \S_2$,
 are equivalent.
\end{theorem}

\begin{proof} It is easy to see that if the triples are isomorphic,
then their associated extensions are equivalent.

Suppose now that the associated extensions are equivalent via the triple of maps $(\underline \alpha, \alpha, \tilde\alpha)$.
Then $\alpha|_{\P_1}= \underline\alpha$, $q_2\circ\alpha=\tilde\alpha \circ q_1$ and $\underline\alpha$ is an extendible isomorphism, say $\underline{\alpha}=\theta|_{\P}$, where $\theta:\N_1\rightarrow \N_2$ is a normal isomorphism with $\theta(\D_1)=\D_2$.
Let $E_i:\M_i\rightarrow \N_i$ be the conditional expectations.  By
Lemma~\ref{Emap},
\[E_2\circ \alpha=(\alpha\circ E_1)|_{\G_1},\dstext{equivalently}
E_2\circ\alpha=(\theta\circ E_1)|_{\G_1}.\]

Let $\psi_1$ be a faithful normal weight on $\N_1$ and let $\psi_2=
\psi_1\circ\theta^{-1}$.  Now let $\phi_i:=\psi_i\circ E_i$.  Then
$\phi_i$ are faithful semi-finite normal weights on $\M_i$.  Let
$(\pi_i, \fH_i, \eta_i)$ be the associated semi-cyclic representations
and let $\fn_i:=\{x\in
\M_i: \phi_i(x^*x)<\infty\}$.  By Lemma~\ref{fndense},
$\spn\{\eta_{\phi_i}(vn): v\in\G\N(\M_i,\D_i)\text{ and
}n\in\fn_{\psi_i}\}$ is dense in $\fH_i$.

Let $n\in\bbN$ and suppose $v_1,\dots, v_n\in\G_1$ and $c_1,\dots ,
c_n\in\fn_{\psi_1}$.  Then $\alpha(v_j)\in \G_2$, and, since
$(\alpha\circ E_1)|_{\G_1}=E_2\circ \alpha$, it follows from the
definition of $\phi_2$ that
\begin{align*} \phi_2\left(\left(\sum_{i=1}^n
\alpha(v_i)\theta(c_i)\right)^*\left(\sum_{i=1}^n
\alpha(v_i)\theta(c_i)\right)\right)
&=\phi_2\left(\sum_{i,j=1}^n\theta(c_i)^*\alpha(v_i^*v_j)\theta(c_j)\right)
\\
&=\psi_2\left(E_2\left(\sum_{i,j=1}^n\theta(c_i)^*\alpha(v_i^*v_j)\theta(c_j)\right)\right)\\
&=\psi_2\left(\sum_{i,j=1}^n\theta(c_i)^*E_2(\alpha(v_i^*v_j))\theta(c_j)\right)\\
&=\psi_2\left(\sum_{i,j=1}^n\theta(c_i)^*\theta(E_1(v_i^*v_j))\theta(c_j)\right)\\
&=\psi_1\left(\sum_{i,j=1}^nE_1(c_i^*v_i^*v_jc_j)\right) \\ &=
\phi_1\left(\left(\sum_{i=1}^n v_ic_i\right)^*\left(\sum_{i=1}^n
v_ic_i\right)\right).
\end{align*} Hence the map
\[\eta_1\left(\sum_{i=1}^n v_ic_i\right) \mapsto
\eta_2\left(\sum_{i=1}^n \alpha(v_i)\theta(c_i) \right)\] extends to a
unitary operator $U:\fH_1\rightarrow \fH_2$.  It is routine to verify
that for $v\in \G_1$, $U\pi_1(v)=\pi_2(\alpha(v))U$.  Therefore the
map $\theta:\M_1\rightarrow \M_2$ given by $\theta(x)=
\pi_2^{-1}(U\pi_1(x)U^*)$ is an isomorphism of $(\M_1,\N_1,\D_1)$ onto
$(\M_2,\N_2,\D_2).$
\end{proof}

\section{Representing an extension}\label{sec: rep G} In this section
we will show how to represent an extension as partial isometries on a
right Hilbert-module.  In Section~\ref{sec: ct from ext} we will show
how this gives rise to a \ct.  
Throughout this section:
\begin{itemize}
\item $\P\hookrightarrow \G\overset{q}\twoheadrightarrow \S$ will be an idempotent
separating extension of the Cartan inverse monoid $\S$ by the
$\N$-Clifford inverse monoid $\P$;
\item $j:\S\rightarrow \G$ will be a fixed order-preserving section (see
  Proposition~\ref{P: order per sec}); and 
\item $\D$ is the von Neumann subalgebra of $\Z(\N)$ generated by
$\E(\P)$.   We will sometimes use the fact that viewed as a \cstaralg, $\D$
is isomorphic to the universal \cstaralg\ $C^*(\widehat{\E(\S)})$ generated by the meet
semilattice $\E(\S)$, see~\cite[Proposition~2.2]{DonsigFullerPittsVNAlExInSe}.
\end{itemize}

We now construct
a right reproducing kernel Hilbert $\N$-module.  We begin by using the
construction of the right reproducing Hilbert $\D$-module as done in
\cite[Section~4.2]{DonsigFullerPittsVNAlExInSe}.
Recall that $j|_{\E(\S)}$ is a complete lattice isomorphism of $\E(\S)$ onto $\E(\P)=\E(\G)\subseteq \D$.
Define $K:\S\times \S\rightarrow \D$ by
\begin{equation*} K(t,s) = j(s^\dag t \wedge 1).
\end{equation*} and for $s\in \S$, define $k_s:\S\rightarrow \D$ by
\begin{equation*} k_s(t)=K(t,s).
\end{equation*} For $d\in\D$ and $s\in\S$, we use $k_sd$ to denote the
map from $\S$ into $\D$ given by $\S\ni t\mapsto k_s(t)d$.  Put
\[\fA_0:=\spn\{k_sd: s\in\S\text{ and }d\in \D\}.\]

Let $u, v\in \fA_0$.  Lemma~4.11 and Proposition~4.12
of~\cite{DonsigFullerPittsVNAlExInSe} show that:
\begin{enumerate}
\item if $u=\sum_{i=1}^nk_{s_i}d_i$ and $v=\sum_{j=1}^n k_{t_j}e_j$,
then the formula
\begin{equation}\label{fancyip} \innerprod{u,v}:=\sum_{i,j=1}^n d_i^*
K(s_i,t_j)e_j
\end{equation} is independent of the choice of the representations for
$u$ and $v$ and determines a well-defined $\D$-valued inner product on
$\fA_0$ which is conjugate linear in the first variable;
\item for every $s\in \S$ and $u\in\fA_0$, $\innerprod{k_s,u}=u(s)$;
\item the completion $\fA_\D$ of $\fA_0$ with respect to this inner
product is a right Hilbert $\D$-module of functions from $\S$ to
$\D$; and
\item $\spn\{k_s: s\in\S\}$ is dense in $\fA_\D$.
\end{enumerate}

Next, we ``fatten'' $\fA_\D$ to incorporate the fact that $\P$ is an
$\N$-Clifford semigroup, not a $\D$-Clifford semigroup as in
\cite{DonsigFullerPittsVNAlExInSe}.  View $\N$ as a right Hilbert
$\N$-module, with $\innerprod{x,y}_\N:=x^*y$.  Define a
$*$-monomorphism $\iota:\D\rightarrow \L(\N)$ by $\iota(d)(x)=dx$
(where $d\in\D$ and $x\in\N$).  Put
\begin{equation}\label{rsdef} \fA:=\fA_\D\otimes_\iota \N,
\end{equation} see \cite[pages~38--44]{LanceHiC*Mo}.  Then $\fA$ is a
right Hilbert $\N$-module.  This is the space on which we shall define
a representation of $\G$.  Note that the inner product on the
algebraic tensor product $\fA_\D\odot_\iota \N$ is
\begin{equation}\label{iprem} \innerprod{\sum_{i=1}^N k_{s_i}\otimes
x_i, \sum_{i=1}^N k_{t_i}\otimes y_i }_{\fA} =\sum_{i,j=1}^N
x_i^*K(s_i,t_j)y_j.
\end{equation} We denote the bounded, adjointable operators on $\fA$
by $\L(\fA)$.

\bigskip We will presently describe the representation of $\G$ on
$\fA$.  First, we need a little more machinery derived from our
extension.  It is usual to describe idempotent-separating extensions
in terms of a cocycle function $\gamma \colon \S \times \S \rightarrow
\P$.  In the case when $\P$ is a abelian this is done explicitly by
Lausch \cite{LauschCoInSe}, leading to a one-to-one correspondence
between extensions and the cohomology group $H^2(\S,\P)$.  In the case
when $\P$ is not abelian D'Alarcao \cite{DAlarcaoIdSeExInSe} has
studied extensions, modelled on the Schreier extensions of groups.
Though no cocycle is explicitly given, the construction again relies on
functions from $\S \times \S$ to $\P$.  In our setting, where we are 
assuming we have an extension $\P \hookrightarrow \G \twoheadrightarrow \S$,
we instead work with a \emph{cocycle-like} function from $\G\times \S$
into $\P$.  This leads to significant computational
simplifications when we define our representation of $\G$.  To our
knowledge, there is not a cohomological description of extensions when
$\P$ is not abelian.  As our cocycle-like function includes
$\G$ a priori, our approach is unlikely to shed further light on that
question.

\begin{definition}\label{D: cocycle like} Define a cocycle-like
function $\sigma \colon \G \times \S \rightarrow \P$ by
\begin{equation*} \sigma(v,s)=j(q(v)s)^\dag vj(s)=j(s^\dag
q(v^\dag))vj(s).
\end{equation*} Since
\begin{equation*} q(\sigma(v,s))=s^\dag q(v^\dag v)s\in \E(\S),
\end{equation*} $\sigma(v,s)\in\P$.  Thus $\sigma$ indeed maps
$\G\times \S$ into $\P$.  Observe also that
\begin{equation}\label{domaincocycle}
\sigma(v,s)^*\sigma(v,s)=j(s^\dag q(v^\dag v)s)=j(s)^\dag v^\dag v
j(s).
\end{equation}
\end{definition}

The following result gives the definition of the representation of
$\G$ in $\L(\fA)$ and  is the analog of
\cite[Theorem~4.16]{DonsigFullerPittsVNAlExInSe} suitable for our
context.  While the outline of the proof is the same as the proof of
\cite[Theorem~4.16]{DonsigFullerPittsVNAlExInSe}, there are
differences.  Due to the importance of the result for our work, we
provide most of the details of the proof.

\begin{theorem}\label{T: Grepdef} For $v\in \G$, $s\in\S$ and $x\in
\N$, the formula,
\[\lambda(v)(k_s \otimes x):= k_{q(v)s} \otimes \cocycle(v,s)x\]
determines a partial isometry $\lambda(v)\in\L(\fA)$.  Moreover,
$\lambda:\G\rightarrow \L(\fA)$ is a one-to-one representation of $\G$
as partial isometries in $\L(\fA)$.
\end{theorem}

\begin{proof} Fix $v\in\G$, and set $r:=q(v)$.  Given $s_1,\dots,
s_N\in\S$, apply 
\cite[Lemma~4.15]{DonsigFullerPittsVNAlExInSe} to obtain $A\subseteq \S$ satisfying:
\begin{enumerate}
\item $0\notin A$;
\item if $a,b\in A$ then $a\wedge b =0$;
\item if $a\in A$ then $a\wedge s_n \in \{0,a\}$ for $1\leq n \leq N$;
and there exists $1\leq n \leq N$ such that $a\wedge s_n=a$;
\item for each $1\leq n \leq N$, $s_n = \bigvee\{a\in A \colon a\leq
s_n\}.$
\end{enumerate}  Choose $c_1,\dots,
c_N\in\N$.

For $a\in A$ and $1\leq m\leq N$, put
\[A_m:=\{b\in A: b\leq s_m\}\dstext{and} c_a:=\sum\{c_n: a\leq
s_n\}.\] Since $A_m \subseteq A$, the elements of $A_m$ are pairwise
meet orthogonal.  Further, $\bigvee A_m=s_m.$

As in the proof of~\cite[Theorem~4.16]{DonsigFullerPittsVNAlExInSe},
\begin{equation}\label{Eq: Grepdef1} \sum_{n=1}^N k_{s_n} \otimes
c_n=\sum_{a\in A} k_{a}\otimes c_a.
\end{equation} Secondly, with routine modifications to the proof of
\cite[Equation~(4.4)]{DonsigFullerPittsVNAlExInSe}, we obtain
\begin{equation}\label{Eq: Grepdef2} \sum_{n=1}^N k_{rs_n}\otimes
\cocycle(v,s_n)c_n=\sum_{a\in A} k_{ra}\otimes\cocycle(v,a)c_a.
\end{equation}

Notice that if $a,b\in A$ are distinct, then $ra$ and $rb$ are
orthogonal, so for $x,y\in\N$,
\begin{align*} \innerprod{k_{ra} \otimes\cocycle(v,a)x, k_{rb}
\otimes\cocycle(v,b)y} &=x^*\cocycle(v,a)^*K(ra,rb)\cocycle(v,b)y\\
&=0\\ &=x^*K(a,b) y\\ &=\innerprod{k_a \otimes x,k_b \otimes y}.
\end{align*} Thus, as $\D\subseteq \Z(\N)$ and using ~\eqref{Eq:
Grepdef2}, then~\eqref{Eq: Grepdef1},
\begin{align*} \innerprod{\sum_{n=1}^N k_{rs_n}
\otimes\cocycle(v,s_n)c_n,\sum_{n=1}^N k_{rs_n}
\otimes\cocycle(v,s_n)c_n}&=\innerprod{\sum_{a\in A} k_{ra}
\otimes\cocycle(v,a)c_a, \sum_{a\in A} k_{ra}
\otimes\cocycle(v,a)c_a}\\ &=\sum_{a\in A}
|c_a|^2\cocycle(v,a)^*j(a^\dag r^\dag r a)\cocycle(v,a)\\ &=\sum_{a\in
A} |c_a|^2j(a^\dag r^\dag r a)\\ &\leq \sum_{a\in A} |c_a|^2j(a^\dag
a)\\ &= \innerprod{\sum_{a\in A} k_{a}\otimes c_a,\sum_{a\in A} k_{a}
\otimes c_a}\\ &=\innerprod{\sum_{n=1}^N k_{s_n} \otimes c_n,
\sum_{n=1}^N k_{s_n}\otimes c_n}.
\end{align*} Therefore,
\begin{equation*} \norm{\sum_{n=1}^N \lambda(v)(k_{s_n}\otimes
c_n)}\leq \norm{\sum_{n=1}^n k_{s_n}\otimes c_n}.
\end{equation*} It follows that we may extend $\lambda(v)$ linearly to
a contractive operator from the algebraic tensor product $\fA_0\odot_\iota\N$ into $\fA$.  Finally extend
$\lambda(v)$ by continuity to a contraction in $\B(\fA)$, the bounded
operators on $\fA$.

We next show that $\lambda(v)$ is adjointable.  As in the proof of the
corresponding equality found in the proof of
\cite[Theorem~4.16]{DonsigFullerPittsVNAlExInSe}, for $s,t\in \S$,
\begin{equation*} \sigma(v,s)^\dag K(rs,t)=\sigma(v^\dag,t)K(s,r^\dag
t).
\end{equation*} Therefore for any $s,t\in\S$ and $x,y\in \N$,
\begin{align*} \innerprod{\lambda(v)(k_s \otimes x), k_t \otimes y} &=
\innerprod{k_{rs} \otimes\cocycle(v,s) x, k_t \otimes y}
=x^*\cocycle(v,s)^* K(rs,t) y\\ &= x^* \cocycle(v^\dag, t) K(s,r^\dag
t)y = x^* K(s,r^\dag t)\cocycle(v^\dag, t)y\\ &=\<k_s \otimes x,
\lambda(v^\dag) (k_t \otimes y)\>.
\end{align*} This equality implies that $\lambda(v)$ is adjointable
and $\lambda(v)^*=\lambda(v^\dag).$

We now show that $\lambda$ is a homomorphism.  Suppose that
$v_1,v_2\in\G$, $x\in\N$ and $s\in\S$.  Then
\begin{align*} \lambda(v_1)(\lambda(v_2)(k_s \otimes x)) &=
\lambda(v_1)(k_{q(v_2)s} \otimes \cocycle(v_2,s)x)\\ &=
k_{q(v_1v_2)s}\otimes \cocycle(v_1,q(v_2)s)\,\, \cocycle(v_2,s)x.
\end{align*} But
\begin{align*} \cocycle(v_1,q(v_2)s)\,\, \cocycle(v_2,s)& =
j(q(v_1)q(v_2)s))^\dag v_1j(q(v_2)s) \,\, j(q(v_2)s)^\dag v_2j(s)\\ &=
j(q(v_1v_2)s))^\dag v_1j(q(v_2)s) \,\, j(s^\dag q(v_2)^\dag) v_2j(s)\\
&= j(q(v_1v_2)s))^\dag v_1 (v_2j(ss^\dag) v_2^\dag) v_2j(s)\\ &=
j(q(v_1v_2)s))^\dag v_1 v_2 v_2^\dag v_2 j(ss^\dag) j(s)\\
&=j(q(v_1v_2)s))^\dag v_1v_2j(s)= \cocycle(v_1v_2,s).
\end{align*} Hence $\lambda(v_1)\lambda(v_2)(k_s \otimes
x)=\lambda(v_1v_2)(k_s \otimes x)$.  As $\spn\{k_s\otimes x:
s\in\S\text{ and } x\in\N\}$ is dense in $\fA$, we conclude that
$\lambda(v_1v_2)=\lambda(v_1)\lambda(v_2).$ It follows that for every
$e\in\E(\G)$, $\lambda(e)$ is a projection.  Furthermore, for $v\in
\G$, $\lambda(v)$ is a partial isometry because
$\lambda(v)^*=\lambda(v^\dag)$.

It remains to show that $\lambda$ is one-to-one.  We first show that
$\lambda|_{\E(\G)}$ is one-to-one.  Suppose $e, f\in\E(\S)$ and
$\lambda(j(e))=\lambda(j(f))$.  Then for every $s\in\S$,
$\sigma(j(e),s)=j(s^\dag e s)\in\D$ and $\sigma(j(f),s) = j(s^\dag f
s)\in \D$.  As the tensor product is balanced,
\begin{align*} k_{es}j(s^\dag es)\otimes I&= k_{es}
\otimes\cocycle(j(e),s)\\ &=\lambda(j(e))(k_s\otimes
I)=\lambda(j(f))(k_s\otimes I)\\ &=k_{fs}
\otimes\cocycle(j(f),s)=k_{fs}j(s^\dag fs)\otimes I,
\end{align*} whence $k_{es} j(s^\dag e s)=k_{fs} j(s^\dag fs)$.
Taking $s=1$ gives $k_ej(e)=k_fj(f)$.  Evaluating these elements of
$\fA_\D$ at $t=1$ gives $j(e)=j(f)$, so $\lambda|_{\E(\G)}$ is
one-to-one.

Now suppose $v_1,v_2\in\G$ and $\lambda(v_1)=\lambda(v_2)$.  Then
\[\lambda(v_1^\dag v_1)=\lambda(v_1^\dag v_2)=\lambda(v_1^\dag
v_2)^*=\lambda (v_2^\dag v_1)=\lambda(v_2^\dag v_2).\] Likewise,
\[\lambda(v_1v_1^\dag) =\lambda(v_1v_2^\dag)
=\lambda(v_2v_1^\dag)=\lambda(v_2v_2^\dag).\] Hence $v_1^\dag
v_1=v_2^\dag v_2$ and $v_1v_1^\dag =v_2v_2^\dag$.  For any
$e\in\E(\S)$, we have
\begin{align*} \lambda(v_1 j(e) v_1^\dag)&= \lambda(v_1 v_1^\dag v_1
j(e) v_1^\dag v_1 v_1^\dag) = \lambda(v_1 v_2^\dag v_2 j(e) v_2^\dag
v_2 v_1^\dag) =\lambda(v_2v_2^\dag v_2 j(e) v_2^\dag v_2 v_2^\dag)\\
&=\lambda(v_2j(e) v_2^\dag).
\end{align*} Hence $v_1j(e)v_1^\dag =v_2j(e)v_2^\dag$.  Since this
holds for every $e\in\E(\S)$ and $\S$ is fundamental, we conclude that
$q(v_1)=q(v_2)$.

Let $e:=q(v_1^\dag v_1)$ and $s:=q(v_1)$.  Since the functions
$\lambda(v_1)k_e$ and $\lambda(v_2)k_e$ agree, we obtain, $k_{s}
j(s)^\dag v_1=k_{s} j(s)^\dag v_2 $.  Evaluating these functions at
$t=s$ gives, $j(s)^\dag v_1=j(s)^\dag v_2$.  Multiplying each side of
this equality on the left by $j(s)$, we obtain $v_1=v_2$.
\end{proof}

Let $\pi:\N\rightarrow \bh$ be a normal representation.  Recall there
is a $*$-representation $\pi_*: \L(\fA)\rightarrow \B(\fA\otimes_\pi
\H)$ given by
\begin{equation}\label{inducedpi} \pi_*(T)(u\otimes \xi)=(Tu)\otimes
\xi.
\end{equation} This representation is strictly continuous on the unit
ball of $\L(\fA)$ and is faithful whenever $\pi$ is faithful~\cite[p.\
42]{LanceHiC*Mo}.  As
in~\cite[Corollary~4.17]{DonsigFullerPittsVNAlExInSe} we have the
following corollary.

\begin{corollary}\label{C: HilbertSpRep} Let $\pi:\N\rightarrow \bh$
be a $*$-representation of $\N$ on the Hilbert space $\H$.  Then
$\lambda_\pi:= \pi_*\circ \lambda$ is a representation of $\G$ by
partial isometries on $\fA\otimes_\pi \H$.  If $\pi$ is faithful, then
$\lambda_\pi$ is one-to-one.
\end{corollary}

\begin{remark}{Remark} Our construction of $\fA\otimes_\pi \H$ depends
upon the order structure of $\S$.  We show presently that the range of
$\lambda_\pi$ will generate a \ct.  When one starts with a \ct\ and
applies this construction to the extension associated to the pair, the
Hilbert space $\fA\otimes_\pi \H$ can be recognized as arising from
the representation associated to a faithful normal weight $\phi$ on
$\M$ such that $\phi\circ E=\phi$.  We will give the formal statement
in Proposition~\ref{H_phi=tenpr} below.
\end{remark}

We close this section with some results which will be needed when
constructing a \ct\ from an extension in Section~\ref{sec: ct from
  ext}.  They will also be used in  Section~\ref{sec: spectral thm} when we study
the $\N$-bimodule structure for a \ct.  Observe that the
construction of the right Hilbert $\D$-module $\fA_\D$ above depends
only upon $\S$ because $j|_{\E(\S)}$ is the inverse of $q|_{\E(\P)}$.
The $\tau_1$-topology described in the following definition has been considered by several authors, see
Section~3.5 of the survey article~\cite{ManuilovTroitskyHiC*W*MoThMo}.

\begin{definition} Let $\fM$ be  a right Hilbert module  over the von
Neumann algebra $\N$.
\begin{enumerate}
\item The \textit{$\tau_1$-topology on $\fM$} is the topology
generated by the seminorms, $\xi\mapsto
\phi(\innerprod{\xi,\xi})^{1/2}$, where $\phi$ is a normal state on
$\N$.
\item
The \textit{$\tau_1$-strict topology on $\L(\fM)$} is the topology  generated
by the seminorms, $T\mapsto \phi(\innerprod{T\xi,T\xi})^{1/2}$ where
$\xi\in \fM$ and $\phi$ is a normal state on $\N$. 
\end{enumerate}
Notice that a net
$T_\alpha\rightarrow T$ in the $\tau_1$-strict topology if and only if
for every $\xi\in \fM$,
\[\innerprod{(T_\alpha-T)\xi,(T_\alpha-T)\xi}\rightarrow 0\] in the
$\sigma$-strong topology of $\N$.
\end{definition}

The following result can be proved directly in the same way as
Proposition~5.2 of~\cite{DonsigFullerPittsVNAlExInSe}, but it is
simpler to apply~\cite[Theorem~4.16 and
Proposition~5.2]{DonsigFullerPittsVNAlExInSe}.

\begin{proposition}\label{Qsdef} For $s\in\S$, the map $\S\ni t\mapsto
k_{s\wedge t}$ extends to a projection $Q_s\in \L(\fA_\D)$ whose range
is $\overline{\spn}\{k_t: t\leq s\}$.  Furthermore, the following
statements hold.
\begin{enumerate}
\item Let $s, t\in \S$.  If $s\wedge t=0$, then $Q_sQ_t=Q_tQ_s=0$; if
$s^\dag t=st^\dag=0$, then $Q_t+Q_s=Q_{s\vee t}$.
\item If $\B\subseteq \S$ is a maximal meet disjoint subset and
$\Lambda$ is the set of all finite subsets of $\B$ directed by
inclusion, then the net $(\sum_{s\in F} Q_s)_{F\in\Lambda}$ converges
$\tau_1$-strictly to the identity operator in $\L(\fA_\D)$.
\end{enumerate}
\end{proposition}

\begin{proof} 
By replacing $\sigma(v,t)$ with the identity operator throughout the proof
of~\cite[Thoerem~4.16]{DonsigFullerPittsVNAlExInSe} (or
Theorem~\ref{T: Grepdef} above) one finds that for every $s\in\S$, there
exists a partial isometry $\lambda_0(s)\in\L(\fA_\D)$ such that
\[\lambda_0(s)k_t=k_{st}.\]   Calculations show that for every $t\in
\S$,
\[\lambda_0(s)P_\D\lambda_0(s)^*k_t=k_{s\wedge t},\] where
$P_\D k_t:=k_{t\wedge 1}$ is the projection from
\cite[Proposition~5.2]{DonsigFullerPittsVNAlExInSe}.  This establishes
the existence of the projection $Q_s$, with the desired range.

The proof of (a) is routine and left to the reader.  Let $\B$ be a
maximal meet disjoint subset of $\S$.  Since the net $(\sum_{s\in F}
Q_s)_{F\in\Lambda}$ is an increasing net of projections, it suffices
to show that for each $t\in \S$, the net
\begin{equation*}
\left(\sum_{s\in F} Q_s
k_t\right)_{F\in\Lambda} \tau_1\text{-converges to } k_t.
\end{equation*}

For $F\in\Lambda$, let $Q_F:=\sum_{s\in F} Q_s$.  For $r\in \S$ and
$s_1, s_2\in F$, $r\wedge s_1$ and $r\wedge s_2$ are disjoint elements
of $A_r:=\{t\in\S: t\leq r\}$, so $(r\wedge s_1)\vee (r\wedge s_2)$ is
defined.  Let $t_F:=\bigvee_{s\in F}(r\wedge s)$.  Then $t_F\leq r$
and $Q_Fk_r=k_{t_F}$.  Denote by $\neg$ the $\mathrm{NOT}$ operation
in the Boolean algebra $\E(\S)$.  We have that
\[\innerprod{Q_Fk_r-k_r, Q_Fk_r -k_r} =\innerprod{k_{t_F}-k_r,
k_{t_F}-k_r}=j(r^\dag r\wedge \neg(t_F^\dag t_F)).\] Let
$b:=\bigvee_{F\in\Lambda} (t_F^\dag t_F)$.  Clearly $b\leq r^\dag r$.
Set $a:=r^\dag r\wedge (\neg b)$.  Then for $s\in \B$,
\[a\wedge s^\dag s = r^\dag r\wedge s^\dag s\wedge (\neg b)\leq r^\dag
r\wedge s^\dag s\wedge \neg((r\wedge s)^\dag (r\wedge s))=0.\] Now
$ra\meet s=ra(r^\dag r\meet s^\dag s)=0$, so that $ra$ is meet
disjoint from every element of $\B$.  By maximality of $\B$, we obtain
$ra=0$, whence $a=0$.  Thus $b=r^\dag r$, from which it follows that
$j(t_F^\dag t_F)$ converges $\sigma$-strongly in $\D$ to $j(r^\dag
r)$.  Therefore, $Q_F$ converges $\tau_1$-strictly to
$I_{\L(\fA_\D)}$.
\end{proof}

We now have the following corollary to Proposition~\ref{Qsdef}.

\begin{corollary} \label{Qsdefcor} The net $Q_F\otimes I_\N$ converges
$\tau_1$-strictly to $I_{\L(\fA)}$.
\end{corollary}

\begin{proof} For $n\in \N$ and $s\in \S$, we have
\begin{align*} \langle(Q_F\otimes I_\N)(k_{s}\otimes n)&-(k_s\otimes
n), (Q_F\otimes I_\N)(k_{s}\otimes n)-(k_s\otimes n)\rangle\\
&=\innerprod{(Q_Fk_s-k_s)\otimes n, (Q_Fk_s-k_s)\otimes n}\\ &=
n^*\innerprod{Q_Fk_s -k_s,Q_Fk_s-k_s}_{\fA_\D} n\\ &=n^*n
\innerprod{Q_Fk_s -k_s,Q_Fk_s-k_s}_{\fA_\D}.
\end{align*} As the last expression tends to zero in the
$\sigma$-strong operator topology on $\N$, the result follows.
\end{proof}

Now let $\pi:\N\rightarrow\bh$ be a faithful normal representation of
$\N$ and for $s\in\S$ define projections on $\fA\otimes_\pi\H$ by
\begin{equation}\label{Pspidef}
P_{s,\pi}:=(Q_s\otimes I_\N)\otimes I_\H.
\end{equation}

\begin{proposition}\label{PsSOT} Let $\B\subseteq \S$ be a maximal
meet disjoint subset.  Then $\sum_{s\in\B} P_{s,\pi}$ converges
strongly to $I\in\B(\fA\otimes_\pi \H)$.
\end{proposition}

\begin{proof} For any $h\in\H$ and $\xi\in \fA$, the map
\[\L(\fA)\ni T\mapsto \innerprod{((T\otimes I_\H)(\xi\otimes
h),(T\otimes I_\H)(\xi\otimes h)}=\innerprod{h,
\pi(\innerprod{T\xi,T\xi}_{\fA})h}_\H^{1/2}\] is a $\tau_1$-strict
seminorm on $\L(\fA)$.  Therefore, if $(T_\alpha)$ is a bounded net in
$\L(\fA)$ which converges $\tau_1$-strictly to $T\in\L(\fA)$,
$(T_\alpha\otimes I_\H)$ converges strongly in $\bh$ to $T\otimes I$.
An application of Corollary~\ref{Qsdefcor} completes the proof.
\end{proof}

\begin{remark}{Remark} Proposition~\ref{PsSOT} is similar to Lemma~\ref{fndense}.
The initial data for Proposition~\ref{PsSOT} is the extension $\P\hookrightarrow
  \G\overset{q}\twoheadrightarrow \S$ and the representation $\pi : \N \to \B(\H)$;
  its conclusion may be interpreted as the statement that $\bigvee_{s\in\S}
  P_{s,\pi}=I_{\fA\otimes_\pi\H}$.  On the other hand,
  Lemma~\ref{fndense} deals with the \ct\
  $(\M,\N,\D)$ and a semi-cyclic representation induced by a suitable weight;
  its conclusion may be interpreted as the statement that $\bigvee_{v\in
    \G\N(\M,\D)}P_v=I_{\H_\phi}$, where $P_v$ is defined in the proof of
  Lemma~\ref{fndense}.  There is a further relation between these results: 
when the extension $\P\hookrightarrow
\G\overset{q}\twoheadrightarrow\S$ is  the extension associated to the
\ct\ $(\M,\N,\D)$,
Proposition~\ref{H_phi=tenpr} below implies that
  for any $v\in\G\N(\M,\D)$, the projections $P_v$ and
  $P_{q(v),\pi_\phi}$ are unitarily equivalent.

\end{remark}

\section{The Cartan Triple Associated to an Extension}\label{sec: ct from ext}

Throughout this section we will consider the extension,
\begin{equation*}
\P\hookrightarrow \G\overset{q}\twoheadrightarrow \S,
\end{equation*}
where $\S$ is a Cartan inverse monoid, and $\P$ is an $\N$-Clifford inverse monoid.
Assume throughout that a fixed order-preserving section $j\colon \S \rightarrow \G$ is given.
Our goal, achieved in Theorem~\ref{exttovn}, is to show how the representation of $\G$ constructed 
in Corollary~\ref{C: HilbertSpRep} gives rise to a \ct\ with associated extension 
$\P\hookrightarrow \G \overset{q}\twoheadrightarrow \S$.
In Theorem~\ref{T: indeppsi} we further show that the extension associated to the \ct\ returns the original extension.

We denote by $\fA$ the right Hilbert $\N$-module as defined in
Equation~\eqref{rsdef}.
Let $\pi$ be a faithful, normal representation of $\N$, and let
\begin{equation*}
\lambda_\pi \colon \G \rightarrow \B( \fA \otimes_\pi \H)
\end{equation*}
be the representation of $\G$ by partial isometries, as constructed in Theorem~\ref{T: Grepdef} and Corollary~\ref{C: HilbertSpRep}.

\begin{definition}\label{ctripledef}
Let
\begin{equation*}
\M\sge := (\lambda_\pi(\G))'',\quad \N\sge:=(\lambda_\pi(\P))'', \text{ and } \D\sge := (\lambda_\pi(\E(\G))''.
\end{equation*}
\end{definition}

We will show that $(\M\sge, \N\sge, \D\sge)$ is a \ct.
The definitions of $\M\sge$, $\N\sge$ and $\D\sge$ depend upon the choice of $\pi$ and, because $\lambda:\G\rightarrow \L(\fA)$ depends on the choice of $j$, $\M\sge$, $\N\sge$  and $\D\sge$ also depend on $j$.
However, we shall see in Theorem~\ref{T: indeppsi} that the \textit{isomorphism class} of $(\M\sge,\N\sge,\D\sge)$ depends only on the extension $\P\hookrightarrow \G\overset{q}\twoheadrightarrow \S$ and not upon $\pi$ or $j$.

The first step is to show that there is a faithful normal conditional expectation from $\M\sge$ onto $\N\sge$.
This will be used to show that $\N\sge = \D\sge^c$ in Proposition~\ref{Dqc=Nq}.
As in \cite{DonsigFullerPittsVNAlExInSe}, the expectation on $\M\sge$ will be induced by the map $s \mapsto s\wedge 1$ on $\S$.

Define $\Delta:\G\rightarrow\P$  by 
\begin{equation*}
\Delta(v):=v j(q(v)\wedge 1),
\end{equation*}
for all $v\in \G$.
First note that
\begin{equation*}
q(\Delta(v))=q(v)(q(v)\wedge 1)= q(v)\wedge 1 \in \E(\S).
\end{equation*}
Thus $\Delta(v)\in\P$ for all $v\in \G$.
Further, if $v\in \P$ then $q(v)\in \E(\S)$, thus
\begin{equation*}
\Delta(v) = v j(q(v)\wedge 1)=v j(q(v))=v.
\end{equation*}

We will show that, given $v\in\G$, the formula,
\[E\sge(\lambda_\pi(v)):=\lambda_\pi(\Delta(v))\]
extends to a faithful conditional expectation $E\sge:\M\sge\rightarrow \N\sge$.  

\begin{remark}{Notation} Here is some notation.
\begin{enumerate}
\item Let $P_\D\in\L(\fA_\D)$ be the projection defined
in Proposition~3.6, so that $P_\D
k_s=k_{s\wedge 1}$.
That is, $P_\D = Q_1$ as defined in
Proposition~\ref{Qsdef}.  Since $\fA= \fA_\D\otimes_\iota\N$, the
tensor product of $P_\D$ with the identity of $\N$ gives a projection
$P\in\L(\fA)$ so that, for $s\in \S$ and $x\in\N$,
\begin{equation}\label{Pdef} P(k_s\otimes x)=k_{s\wedge 1}\otimes x.
\end{equation}

\item For $x\in \N$,  and $y=\sum_{i=1}^n k_{s_i}\otimes n_i\in
  \fA_\D\odot \N$, $\norm{\sum_{i=1}^n k_{s_i}\otimes xn_i}_{\fA}\leq
  \norm{x}\norm{y}_\fA$.   It follows that the map $k_s\otimes
  n\mapsto k_s\otimes xn$ extends to a bounded linear map $\pil(x)$ on
  $\fA$.  A computation shows that $\pil(x)$ is adjointable, so 
  there exists a faithful $*$-representation $\pil:\N\rightarrow
  \L(\fA)$.   Tensoring with the identity map,  we obtain a faithful normal
representation $\pil_*=\pil\otimes I$ of $\N$ on $\B(\fA\otimes_\pi
\H)$.  To be explicit, for $x\in \N$, $\pil_*(x)$ is defined on
elementary tensors $(k_s\otimes n\otimes \xi) \in\fA\otimes_\pi \H$ by
\begin{equation}\label{fatN} \pil_*(x)(k_s\otimes n\otimes \xi)
=k_s\otimes xn\otimes \xi=k_s\otimes I_\N\otimes \pi(xn)\xi.
\end{equation}
\end{enumerate}
\end{remark}

\begin{lemma}\label{Nqcom} 
For $s\in\S$, let $Q_s$ be the projection
on $\L(\fA_\D)$ defined in Proposition~\ref{Qsdef} and let
$P_{s,\pi}:=Q_s\otimes I_\N\otimes I_\H\in \B(\fA\otimes_\pi\H)$ be
the projection 
defined in Equation~\eqref{Pspidef}.
The following statements hold.
\begin{enumerate}
\item For $v\in\P$, $s\in\S$, $n\in \N$ and
  $\xi\in \H$, 
\[\lambda_\pi(v)(k_s\otimes n\otimes \xi)=k_s\otimes j(s)^*vj(s)
n\otimes \xi.\]
\item $P_{s,\pi}\in \N\sge'$ and for every
$v\in\P$, $\lambda_\pi(v)P_{s,\pi}=\pil_*(j(s)^*vj(s))P_{s,\pi}$.
\end{enumerate}
\end{lemma}

\begin{proof} Since $v\in\P$, $q(v)=q(vv^*)$, so $\sigma(v,s)=j(s^\dag
q(v)^\dag)vj(s)=j(s)^*vj(s)$.  Therefore,
\begin{align*} \lambda_\pi(v)(k_s\otimes n\otimes
\xi)&=k_{q(v)s}\otimes j(s)^*vj(s)n\otimes \xi=k_{ss^\dag
q(v)s}\otimes j(s)^*vj(s)n\otimes \xi\\  &=k_sj(s^\dag q(v)s)\otimes
j(s)^*vj(s)n\otimes \xi =k_s\otimes j(s^\dag q(v)s)j(s)^*vj(s)n\otimes
\xi\\ &= k_s\otimes j(s)^*vj(s)n\otimes \xi,
\end{align*}  where the third equality follows from
\cite[Corollary~4.9]{DonsigFullerPittsVNAlExInSe}.  This gives part
(a) and shows $\ran(P_{s,\pi})$ is invariant for
every element of $\lambda_\pi(\P)$.  Thus, $\ran(P_{s,\pi})$ is invariant for $\N\sge$.  As
$\N\sge$ is a $*$-algebra, $P_{s,\pi}$ reduces $\N\sge$, whence
$P_{s,\pi}\in \N\sge'$.  

Now suppose $t\in\S$, $n\in\N$ and $\xi\in\H$. Note that $j(s)j(s^\dag t\wedge
1)=j(s(s^\dag t\wedge 1))=j(s\wedge t)$.   Then, again using
\cite[Corollary~4.9]{DonsigFullerPittsVNAlExInSe}, 
\begin{equation}\label{Nqcom1} P_{s,\pi}(k_t\otimes n\otimes \xi)=k_{s\wedge t}\otimes
n\otimes\xi=k_s\otimes j(s^\dag t\wedge 1)n\otimes \xi.
\end{equation}
A computation using part (a) and~\eqref{Nqcom1} now gives the formula in part (b).
\end{proof}

With the obvious modifications to the proof of
\cite[Proposition~5.2]{DonsigFullerPittsVNAlExInSe}, we obtain the
following.
\begin{lemma}\label{PProp} With $P$ defined as in
Equation~\eqref{Pdef}, the following properties hold:
\begin{enumerate}
\item $\ran P=\overline{\spn}\{k_e \otimes x: e\in\E(\S)\text{ and }
x\in \N\}$; and
\item for $v\in \G$,
\[P\lambda(v)P=\lambda(\Delta(v))P.\]
\end{enumerate}
\end{lemma}

Modifications to the proof
of~\cite[Proposition~5.3]{DonsigFullerPittsVNAlExInSe} yield the
following result.

\begin{lemma}\label{Vdef} The map $V:\H\rightarrow \fA\otimes_\pi \H$
given by $V\xi= (k_1\otimes I_\N)\otimes \xi$ is an isometry.
Moreover, the following properties hold:
\begin{enumerate}
\item for $s\in\S$, $x\in\N$ and $\xi\in\H$, $V^*(k_s\otimes x\otimes
\xi)=\pi(j(s\wedge 1)x) \xi$;
\item $VV^*=\pi_*(P)$, where $\pi_*: \L(\fA)\rightarrow \B(\fA\otimes_\pi
\H)$ is defined by $\pi_*(T)(u\otimes \xi)=(Tu)\otimes
\xi$;
\item for $v\in\G$, $V^*\lambda_\pi(v) V=\pi(\Delta(v))$.
\end{enumerate}
\end{lemma}

\begin{lemma}\label{v*Mqv} We have
\[V^*\M_q V=\pi(\N)=V^*\N\sge V.\]
\end{lemma}
\begin{proof} Lemma~\ref{Vdef}(c) shows that for $x\in\M_q$,
$V^*xV\in\pi(\N)$, so $V^*\M_q V\subseteq \pi(\N)$.  On the other
hand, for $v\in \P$ we have
\begin{equation}\label{v*Mqv1} V^*\lambda_\pi(v)
V=\pi(\Delta(v))=\pi(v),
\end{equation} so $V^*\N\sge V\subseteq \pi(\N)$.  Since every element
of $\N$ is a linear combination of at most four elements of $\P$, we
obtain the result.
\end{proof}
 
Thus, the map $\M_q \ni x\mapsto \pi^{-1}(V^*xV)$ is a normal,
completely positive contraction of $\M_q$ onto $\N$.  We now show this
map gives an isomorphism of $\N\sge$ onto $\N$.
\begin{lemma}\label{isovn} The map $\alpha:\N\sge\rightarrow \N$ defined by
$\alpha(x)=\pi^{-1}(V^*xV)$ is a normal isomorphism of
$\N\sge$ onto $\N$.
\end{lemma}

\begin{proof} The definition of $\alpha$ shows it is normal.  Next we
  show that $\alpha$ is a homomorphism.  For $v_1, v_2\in\P$,
  Lemma~\ref{Vdef}(c) gives
\[V^*(\lambda_\pi(v_1))V
V^*\lambda_\pi(v_2)V=\pi(v_1)\pi(v_2)=\pi(v_1v_2)=V^*\lambda_\pi(v_1v_2)V.\]
Thus $\alpha$ is multiplicative on $\lambda_\pi(\P)$.   It follows
that $\alpha$ is multiplicative on
$\spn(\lambda_\pi(\P))$.  As multiplication is $\sigma$-strongly
continuous on bounded sets, the Kaplansky density theorem ensures
$\alpha$ multiplicative.  Lemma~\ref{v*Mqv} now shows $\alpha$ is a $*$-epimorphism.

It remains to show $\alpha$ is one-to-one.  To do this, we show
$\alpha$ is isometric on $\spn(\lambda_\pi(\P))$.  Suppose $n\in\bbN$,
$c_j\in\bbC$, $v_j\in\P$ and $x=\sum_{j=1}^nc_j\lambda_\pi(v_j)$.  Put
$y=\alpha(x)$ so that $y=\sum_{j=1}^n c_jv_j$ by Equation~\ref{v*Mqv1}.  By Lemma~\ref{Nqcom}, for each
$s\in\S$,
\[xP_{s,\pi}=\pil_*(j(s)^*yj(s))P_{s,\pi},\] and hence
$\norm{xP_{s,\pi}}\leq \norm{y}$.  Now suppose $\B$ is a maximal
meet-disjoint subset of $\S$.  Then for distinct $s, t\in \B$,
$P_{s,\pi}$ and $P_{t,\pi}$ are orthogonal projections.  By
Proposition~\ref{PsSOT},
\[\norm{x}=\sup_{s\in\B}\norm{xP_{s,\pi}}\leq
\norm{y}=\norm{\alpha(x)}\leq\norm{x}.\] So $\alpha$ is isometric on
$\spn(\lambda_\pi(\P))$.
\end{proof}
 
At last, we can define the conditional expectation $E\sge$.

\begin{proposition}\label{ceexist} The formula,
\begin{equation}\label{cedef} E\sge(x):=\alpha^{-1}(\pi^{-1}(V^*xV))
\end{equation} gives a faithful normal conditional expectation of
$\M\sge$ onto $\N\sge$.  Furthermore, for $v\in\G$,
\begin{equation}\label{ceexist1}
E\sge(\lambda_\pi(v))=v\lambda_\pi(\Delta(v)).
\end{equation}
\end{proposition}

\begin{proof} Lemmas~\ref{v*Mqv} and~\ref{isovn} imply $E\sge$ is a normal
conditional expectation of $\M\sge$ onto $\N\sge$.  It remains only to
establish that $E\sge$ is faithful.  The proof that $E\sge$ is faithful is
modeled on the proof of~\cite[Proposition~5.9]{DonsigFullerPittsVNAlExInSe}.

Let $\C$ denote the center of $\M\sge$.  We claim that $E\sge|_\C$ is
faithful.  Let $x\in\C$ and suppose $E\sge(x^*x)=0$.  The definition of
$E\sge$ from Equation~\ref{cedef} shows that $V^*x^*xV=0$ and hence
$xV=0$.  Notice that $\sigma(j(s),1)=j(s^\dag s)$ so that for
$n\in\N$, $\lambda(j(s))(k_1\otimes n) =k_s j(s^\dag s)\otimes
n=k_s\otimes n$
(see~\cite[Corollary~4.9]{DonsigFullerPittsVNAlExInSe}).  Hence for
$s\in\S$, $n\in \N$ and $\xi\in \H$,
\begin{equation*} x(k_s\otimes n\otimes \xi)=
x\lambda_\pi(j(s))(k_1\otimes n\otimes \xi) =\lambda_\pi(j(s))
x(k_1\otimes n\otimes \xi) =\lambda(j(s))xV\pi(n)\xi=0.
\end{equation*} Since the span of such vectors is a dense subspace of
$\fA\otimes_\pi \H$, we conclude that $x=0$.

Let $\fJ:=\{x\in\M\sge: E\sge(x^*x)=0\}$.  Then $\fJ$ is a left ideal of
$\M\sge$.  Compute as in the second part
of~\cite[Lemma~5.8]{DonsigFullerPittsVNAlExInSe} to find that for
$x\in\fJ$ and $v\in \G$,
\[E\sge(\lambda_\pi(v)^*x^*x\lambda_\pi(v))=\lambda_\pi(v)^*E\sge(x^*x)\lambda_\pi(v)=0.\]
Thus, $x\lambda_\pi(v)\in\fJ$.  It now follows that $\fJ$ is a
two-sided ideal of $\M\sge$ as well.  Since $\fJ$ is weak-$*$-closed,
by \cite[Proposition~II.3.12]{TakesakiThOpAlI}, there is a projection
$Q\in\C$ such that $\fJ=Q\M\sge.$ As $Q\in\fJ$ and $E\sge|_\C$ is
faithful, we obtain $Q=0$.  Thus $\fJ=(0)$, that is, $E\sge$ is
faithful.  The equality~\eqref{ceexist1} follows from
Lemma~\ref{Vdef}.
\end{proof}

\begin{proposition}\label{Dqc=Nq} The algebra $\N\sge$ is the relative
commutant of $\D\sge$ in $\M\sge$.  That is, $\N\sge = \D\sge^c$.
\end{proposition}

\begin{proof} Notice that $v\in\G$ commutes with every element of
$\E(\G)$ if and only if $v\in \P$.  Since $\lambda_\pi$ is a
isomorphism of $\G$ onto $\lambda_\pi(\G)$, we obtain
$\lambda_\pi(\G)\cap \D\sge^c =\lambda_\pi(\P)$.  Therefore,
$\N\sge\subseteq \D\sge^c$.

Take $x\in \D\sge^c$.  Suppose $w\in \lambda_\pi(\G)$ satisfies $w^2
=0$.  Then
\begin{equation}\label{dqc=Nq1} E\sge(w^*x) = w^*w E\sge(w^*x) =
  E\sge(w^*x) w^*w = E\sge(w^*x w^*w) = E\sge((w^*)^2wx) = 0.
\end{equation}

Now choose an arbitrary $v\in \lambda_\pi(\G)$.  Our goal is to show
that (again with $x\in\D\sge^c$) \begin{equation}\label{dqc=Nq2} 
E\sge(v^*x) = E\sge(v^*)E\sge(x).\end{equation}  As in Lemma~\ref{Emap}, $v$
defines a map $\theta_v$ on $\D\sge$ ($d\mapsto vdv^*$).  By Frol\`{i}k's Theorem (see
\cite[Proposition~2.11a]{PittsStReInI}) there are orthogonal
projections $e_0,\ e_1,\ e_2, e_3 \in \D\sge$ such that
\[ v =\sum_{k=0}^3 ve_k, \]
$\theta_v|_{\D e_0} = \id|_{\D e_0}$, and $\theta_v(e_k)e_k = 0$ for
$k = 1,2,3$.

As $e_0$ is the largest projection in $\D_q$ on which $\theta_v|_{\D
e_0} = \id_{\D e_0}$, it follows that $ve_0 = E\sge(v)$.  Also, as
$\theta_v(e_k)e_k = 0$, it follows that $(v e_k)^2 = 0$, for
$k=1,2,3$.  By \eqref{dqc=Nq1}, for $k=1,2,3$, $E\sge(e_k v^* x) = 0$.  Thus
\[ E\sge(v^*x) = \sum_{k=0}^3 E\sge(e_k v^* x) = E\sge(e_0 v^* x) \\
= E\sge(v^*)E\sge(x), \] so~\eqref{dqc=Nq2} holds.

Since $\lambda_\pi(\G)$ spans a weak$^*$-dense subset of $\M_q$, it
follows that for $x\in\D_q^c$ we have
\[ E\sge(x^*x) = E\sge(x^*)E\sge(x).\]
Hence, if $x\in \D\sge^c$ we have
$$ E\sge((x-E\sge(x))^*(x-E\sge(x)) = E\sge(x-E\sge(x))^* E\sge(x-E\sge(x)) = 0. $$
Since $E\sge$ is faithful, it follows that $x = E\sge(x) \in \N_q$.
\end{proof}

Proposition~\ref{ceexist} and Proposition~\ref{Dqc=Nq} now immediately
give the first main theorem of this section.

\begin{theorem}\label{exttovn} $(\M\sge,\N\sge,\D\sge)$ is a \ct.
\end{theorem}

The second main theorem of this section is 
that the extension associated to $(\M\sge,\N\sge,\D\sge)$ gives back
the original extension $\P\hookrightarrow\G\overset{q}\twoheadrightarrow \S$.

\begin{theorem}\label{T: indeppsi} Let $\P$ be a $\N$-Clifford inverse
monoid and suppose $\P\hookrightarrow\G\overset{q}\twoheadrightarrow
\S$ is an extension of the Cartan inverse monoid $\S$ by $\P$.  Let
$(\M\sge, \N\sge,\D\sge)$ be the \ct\ constructed in
Theorem~\ref{exttovn}.  The extension associated to $(\M\sge,\N\sge,
\D\sge)$ is equivalent to the extension
\begin{equation*}\label{F: indeppsi1}
\P\hookrightarrow\G\overset{q}\twoheadrightarrow \S
\end{equation*} from which $(\M\sge,\N\sge,\D\sge)$ was constructed.

Moreover, the isomorphism class of $(\M\sge,\N\sge,\D\sge)$ depends
only upon the equivalence class of the extension (and not on the
choice of representation $\pi$ or section $j$).
\end{theorem}

\begin{remark}{Remark}\label{ArvesonCETh} In the proof of
  Theorem~\ref{T: indeppsi} and also in the proof of
  Theorem~\ref{plentyN} below, we shall utilize a result of
Arveson,~\cite[Theorem~6.2.2]{ArvesonAnOpAl}.
In~\cite{ArvesonAnOpAl}, Arveson makes the blanket assumption that all
Hilbert spaces are separable (see~\cite[Section~1.2]{ArvesonAnOpAl}).
However the proof of~\cite[Theorem~6.2.2]{ArvesonAnOpAl} does not
require separability.
\end{remark}

\begin{proof} The argument below is a modification of the proof
of~\cite[Theorem~5.12]{DonsigFullerPittsVNAlExInSe}.  Let $R_M$ and
$R_{M,\pi}$ be the Munn congruences for $\G$ and $\lambda_\pi(\G)$
respectively.  Since $\lambda_\pi$ is an isomorphism of $\G$ onto
$\lambda_\pi(\G)$, $(v,w)$ belongs to $R_M$ if and only if
$(\lambda_\pi(v),\lambda_\pi(w))$ belongs to $R_{M,\pi}$.  Let
$q_\pi:\lambda_\pi(\G)\rightarrow \lambda_\pi(\G)/R_{M,\pi}$ be the
quotient map.   The fact that $\S$ is fundamental implies that $\tilde{\lambda}_\pi:= q_\pi\circ
\lambda_\pi\circ j$ is a multiplicative map of $\S$ onto
$\lambda_\pi(\G)/R_{M,\pi}$.  In fact, $\tilde{\lambda}_\pi$ is an
isomorphism satisfying $\tilde{\lambda}_\pi\circ
q=q_\pi\circ \lambda_\pi$, and furthermore, $\lambda_\pi|_\P$ is an
isomorphism of $\P$ onto $\lambda_\pi(\P)$.  Let $v\in \P$.  By
Equation~\eqref{v*Mqv1} (see Lemma~\ref{v*Mqv}),
$V^*\lambda_\pi(v)V=\pi(v)$.  Thus in the notation of
Lemma~\ref{isovn}, $\alpha^{-1}(v)=\lambda_\pi(v)$.  Therefore
$\lambda_\pi|_\P=\alpha^{-1}|_\P$.  It is now clear that the
extensions
\[\P\hookrightarrow\G\overset{q}\twoheadrightarrow \S\] and
\[\lambda_\pi(\P) \hookrightarrow \lambda_\pi(\G)
\overset{q_\pi}\twoheadrightarrow \widetilde{\lambda}_\pi(\S)\] are equivalent.  For
later use, note that in particular, $\lambda_\pi|_{\E(\G)}$ is an
isomorphism of $\E(\G)$ onto $\E(\lambda_\pi(\G))$.

Our next task is to show that
\begin{equation}\label{indeppsi2} \lambda_\pi(\G) = \G\N(\M_q,\D_q).
\end{equation} It will then follow immediately that $\lambda_\pi(\P)
\hookrightarrow \lambda_\pi(\G) \overset{q_\pi}\twoheadrightarrow
\widetilde{\lambda}_\pi(\S)$ is the extension associated to
$(\M_q,\N_q,\D_q)$.

\textit{Claim 1: If $u\in\G\N(\M_q,\D_q)$, then $uE\sge(u^*)$ is a
projection in $\D\sge$, and
\begin{equation}\label{E: saE} uE\sge (u^*)=E\sge (uE\sge (u^*))=E\sge
(u)E\sge (u^*).
\end{equation}} 

Let $\mean$ be an invariant mean on the abelian group $\U(\D\sge)$.
By~\cite[Theorem~6.2.2]{ArvesonAnOpAl},
\[ uE\sge(u^*)=\mean_{g\in\U(\D\sge)} (ugu^*)g^*\in\D\sge.\]
 Next,
\[uE\sge (u^*)uE\sge (u^*)=uE\sge (u^* uE\sge (u^*))=uu^*u E\sge
((E\sge (u^*))=uE\sge (u^*),\] so $uE\sge (u^*)$ is a projection in
$\D\sge$.  
The equality \eqref{E: saE} is now obvious, so Claim 1 holds.

By construction, $\lambda_\pi(\G)\subseteq \G\N(\M_q,\D_q)$.  To
establish the reverse inclusion, fix $v\in\G\N(\M_q,\D_q)$; without
loss of generality, assume $v\neq 0$.

\textit{Claim 2: There exists $p\in\lambda_\pi(\E(\G))$ such that: a)
$vp\in \lambda_\pi(\G)$, b) $p\leq v^*v$, and c) $vp\neq 0$.}

Since
$\lambda_\pi(\G)''=\M_q$, it follows (as in the proof of
\cite[Proposition~1.3.4]{CameronPittsZarikianBiCaMASAvNAlNoAlMeTh})
that there exists $w\in \lambda_\pi(\G)$ such that $wE\sge(w^*v)\neq 0$.
Let $p=v^*wE\sge(w^*v).$ By Claim 1, $p\in\D_q$ is a projection, so in
particular, $p\in\lambda_\pi(\E(\G))$.  It is evident that $p\leq
v^*v$.  By~\eqref{E: saE},
\begin{equation*}\label{pileft} E\sge (v^*w)E\sge (w^*v) =p,
\end{equation*} so $x:=E\sge(w^*v)$ is a partial isometry in $\N\sge$ with
source projection $p\in\D\sge$.   On the other hand, let
$p':=w^*vE\sge(v^*w)$.  Claim 1 gives $p'\in\D\sge$ and
$p'=E\sge(w^*v)E\sge(v^*w)$.  We have thus shown that both the source
and range projections for $x$ belong to $\D\sge\subseteq\Z(\N\sge)$.
Therefore,
\[p=x^*x=x^*(xx^*)x=(x^*x)(xx^*)=x(x^*x)x^*=xx^*=p'.\] 
Hence 
$E\sge (w^*v)$ is a partial isometry in $\N_q$ whose
source and range projections  both equal $p\in\D\sge$.  Thus,
$E\sge(w^*v)\in \G\N(\N\sge,\D\sge)=\lambda_\pi(\P)$.   This gives $wE\sge
(w^*v)\in\lambda_\pi(\G)$.  Since $E\sge (w^*v)=w^*v(v^*wE\sge
(w^*v)),$ we obtain,
\[ 0\neq wE\sge (w^*v)=w (w^*v(v^*wE\sge (w^*v))) = vv^*wE\sge
(w^*v)=vp.\] Thus Claim 2 holds.

Now argue exactly as in the proof
of~\cite[Theorem~5.12]{DonsigFullerPittsVNAlExInSe} to conclude that
$v\in\lambda_\pi(\G)$.  Therefore, we have shown that
$\lambda_\pi(\G)=\G\N(\M\sge,\D\sge)$.  Hence 
\begin{equation*} \lambda_\pi(\P) \hookrightarrow \lambda_\pi(\G)
\overset{q_\pi}\twoheadrightarrow{q_\pi} \widetilde{\lambda}_\pi(\S)
\end{equation*} is the extension for $(\M_q,\N_q,\D_q).$

Suppose that $\tilde \pi$ is a faithful normal representation of $\N$
and $\tilde j:\S\rightarrow \G$ is an order preserving section for
$q$.  Let $(\tilde\M_q,\tilde\N\sge,\tilde\D\sge)$ be the \ct\
constructed using $\tilde\pi$ and $\tilde j$ as in
Theorem~\ref{exttovn}.  Then the previous paragraphs show that the
extensions associated to $(\M_q,\N_q, \D_q)$ and $(\tilde
\M_q,\tilde\N_q,\tilde\D_q)$ are equivalent extensions.  By
Theorem~\ref{T: same data}, $(\M_q,\N_q,\D_q)$ and
$(\tilde\M_q,\tilde\N_q,\tilde\D_q)$ are isomorphic \ct s.  The proof
is now complete.
\end{proof}

Let $(\M,\N,\D)$ be a \ct\  and let $\phi$ be a faithful normal
semi-finite weight on $\M$ satisfying $\phi \circ E = \phi$.  We end
this section by relating the semi-cyclic representation $(\pi_\phi,
\H_\phi, \eta_\phi)$ and the reproducing kernel Hilbert $\N$-module
$\fA \otimes\N$ constructed from the extension $\P\hookrightarrow\G\overset{q}\twoheadrightarrow \S$ associated to $(\M,\N,\D)$. 

\begin{proposition}\label{H_phi=tenpr} Let $(\M,\N,\D)$ be a \ct,
suppose $\psi$ is a faithful, normal semi-finite weight on $\N$, and
put $\phi:=\psi\circ E$.  Let $(\pi_\psi, \H_\psi, \eta_\psi)$ and
$(\pi_\phi,\H_\phi,\eta_\phi)$ be the semi-cyclic representations of
$\N$ and $\M$ associated with $\psi$ and $\phi$ respectively.  Let
\[\P\hookrightarrow\G\overset{q}\twoheadrightarrow \S\] be the extension associated
to $(\M,\N,\D)$ and let $j:\S\rightarrow \G$ be an order-preserving
section for $q$.  For $s\in\S$, $n\in\N$ and $x\in\fn_\psi$,
$j(s)nx\in \fn_\phi$, and the map
\[(\fA_\D\otimes_\iota \N)\otimes_{\pi_\psi}\H_\psi\ni k_s\otimes
n\otimes \eta_\psi(x)\mapsto \eta_\phi(j(s)nx)\in\H_\phi\] extends to
a unitary operator $W:\fA\otimes_{\pi_\psi}\H_\psi\rightarrow \H_\phi$
such that for every
$v\in\G$, \[W\lambda_{\pi_\psi}(v)W^*=\pi_\phi(v).\]
\end{proposition}

\begin{proof} For $i=1,2$, let $s_i\in\S$, $n_i\in\N$ and
$x_i\in\fn_\psi$.  Then
\[\phi((j(s_i)n_ix_i)^*(j(s_i)n_ix_i))=\psi(x_i^*n_i^*E(j(s_i)^*j(s_i))n_ix_i)\leq
\norm{n_i}^2\psi(x_i^*x_i), \] so $j(s_i)n_ix_i\in \fn_\phi$.  Recall
that for any $v\in\G\N(\M,\D)$, $E(v)=vj(q(v)\wedge 1)$.  In
particular, using Proposition~\ref{P: order per sec}, we have
\[E(j(s_1)^*j(s_2))=j(s_1)^* j(s_2)j(s_1^\dag s_2\wedge 1)=j(s_1^\dag
s_2 \wedge 1).\] So
\begin{align*} \innerprod{(k_{s_1}\otimes n_1\otimes \eta_\psi(x_1)),
(k_{s_2}\otimes n_2\otimes \eta_\psi(x_2))} &=
\psi(x_1^*n_1^*j(s_1^\dag s_2\wedge 1)n_2x_1)\\
&=\psi(x_1^*n_1^*E(j(s_1)^*j(s_2))n_2x_1)\\
&=\innerprod{\eta_\phi(j(s_1)n_1x_1),\eta_\phi(j(s_2)n_2x_2)}.
\end{align*}

As every element in $\spn\{k_s\otimes n\otimes x: s\in\S, n\in \N, x\in
\fn_\psi\}$ can be written as $\sum_{a\in A} k_{a}\otimes n_a\otimes
x_a$ where $A\subseteq \S$ is a finite pairwise meet disjoint set,
$\{n_a: a\in A\}\subseteq \N$ and $\{x_a: a\in A\}\subseteq \fn_\psi$,
it follows that $k_s\otimes n\otimes x\mapsto \eta_\phi(j(s)nx)$
extends to an isometry $W:\fA\otimes_{\pi_\psi}\H_\psi\rightarrow
\H_\phi$.  By Proposition~\ref{fndense},
$\spn\{\eta_\phi(j(s)nx):s\in\S, n\in\N, x\in \fn_\psi\}$ is dense in
$\H_\phi$, so $W$ is a unitary operator.

If $v\in\G\N(\M,\D), s\in\S, n\in \N$ and $x\in\fn_\psi$,
\begin{align*} W\lambda_{\pi_\psi}(v)(k_s\otimes n\otimes x)
&=W(k_{q(v)s}\otimes \sigma(v,s)n\otimes x)\\ &=
\eta_\phi(j(q(v)s)\sigma(v,s)nx)\\ &=\eta_\phi(vj(s)nx)=\pi_\phi(v)
W(k_s\otimes n\otimes x).
\end{align*} Thus, $W\lambda_{\pi_\psi}(v)W^*=\pi_\phi(v).$
\end{proof}

\section{The spectral theorem for bimodules and Aoi's
theorem}\label{sec: spectral thm} Throughout this section,
$(\M,\N,\D)$ will be a \ct\ with associated extension
\[\P\hookrightarrow\G\overset{q}\twoheadrightarrow \S,\] and $j:\S\rightarrow\G$
will be a fixed choice of an order-preserving section for $q$.  The
goal in this section is study the $\N$-bimodules in $\M$.  Recall the
following definition from \cite{DonsigFullerPittsVNAlExInSe}.

\begin{definition}\label{def: spec set} A subset $A$ of a Cartan
inverse monoid $\S$ is a \emph{spectral set} if
\begin{enumerate}
\item $s \in A$ and $t \leq s$ implies that $t \in A$; and
\item if $\{s_i\}_{i\in I}$ is a pairwise orthogonal family in $A$,
then $\bigvee_{i\in I} s_i \in A$.
\end{enumerate}
\end{definition} In Theorem~\ref{spthbi} we prove a Spectral Theorem
for Bimodules.  Will show a one-to-one correspondence between the
spectral sets in $\S$ and a large class of weak-$*$ closed
$\N$-bimodules: the Bures-closed $\N$-bimodules (see
Definition~\ref{BuresTop}).  We go on to study the intermediate von
Neumann algebras $\N \subseteq \L \subseteq \M$, giving a
generalization of Aoi's Theorem \cite{AoiCoEqSuInSu} in
Theorem~\ref{Aoi}.  Several of these theorems require that the \ct\ is
a full \ct.

\subsection{$\N$-bimodules} We begin by showing that weak-$*$ closed
$\N$-bimodules give rise to non-empty spectral sets.  In particular,
Theorem~\ref{plentyN} shows that when $(\M,\N,\D)$ is a full \ct , any
weak-$*$ closed $\N$-bimodule contains an abundance of elements of
$\G\N(\M,\D)$.  Example~\ref{notfull} below gives a simple example showing fullness is
necessary.

\begin{theorem} \label{plentyN} Let $(\M,\N,\D)$ be a full \ct.
Suppose $(0)\neq\B\subseteq \M$ is a weak-$*$-closed $\N$-bimodule.
Then
\[\{0\}\neq \G\N(\M,\D)\cap \B.\] In fact, for every $x\in \B$ and
$v\in\G\N(\M,\D)$, $vE(v^*x)$ is a linear combination of at most four
elements of $\G\N(\M,\D)\cap \B$.
\end{theorem}

\begin{proof} Let $x\in\B$ be non-zero. Since $\M$ is the weak-$*$
closed span of $\G\N(\M,\D)$, there exists $v\in\G\N(\M,\D)$ such that
$E(v^*x)\neq 0$, and hence 
$vE(v^*x)\neq 0$.  By \cite[Theorem~6.2.2]{ArvesonAnOpAl} (see
Remark~\ref{ArvesonCETh} above), for any
$y\in \M$,
\[E(y)=\mean_{U\in\U(\D)} U^*yU.\] Therefore,
\[vE(v^*x)=\mean_{U\in\U(\D)} (vU^*v^*)xU \in \B.\]

Let $J$ be the weak-$*$
closed, two-sided ideal in $\N$ generated by $E(v^*x)$.  For any $n_1,
n_2 \in \N$ we have
\begin{align*} v (n_1 E(v^*x) n_2) = (v n_1 v^*) (vE(v^*x)) n_2 \in
\B.
\end{align*} Since $\B$ is weak-$*$ closed, it follows that
$vJ\subseteq \B$.  Let $p\in \Z(\N)=\D$ be such that $J=p\N$.  Then,
$vp\in\B\cap \G\N(\M,\D)$.  Since $vE(v^*x)=vpE(v^*x)$, $0\neq vp$.

Since $p\N$ is a von Neumann algebra (with unit $p$), $E(v^*x)$ is a
linear combination of four unitary elements of $p\N$.  As
$p\in\Z(\N)$, $\U(p\N)=\{pw: w\in \U(\N)\}$.  Also, for any unitary
$w\in\U(\N)$ we have $vpw\in\G\N(\M,\D)$.  Thus, $vE(v^*x)$ is a
linear combination of at most four elements of $\G\N(\M,\D)$.
\end{proof}

The following corollary of the proof of Theorem~\ref{plentyN} will be
needed in the sequel.

\begin{corollary}\label{plentyNCor} Let $(\M,\N,\D)$ be a (not
necessarily full) \ct.  For any $x\in \M$ and $v\in\G\N(\M,\D)$,
$vE(v^*x)$ belongs to the weak-$*$ closed $\D$-bimodule generated by
$x$.
\end{corollary}

\begin{example}\label{notfull} Let $\M$ be any von Neumann algebra
with non-trivial center.  Then $(\M,\M,\bbC I)$ is a \ct\ which is not
full.  Let $p$ be a central projection in $\M$ with $0<p<I$.  Then $\M
p$ is a weak-$*$ closed $\M$-bimodule.  However $\M p \cap
\G\N(\M,\bbC I) = \{0\}$.  Thus, the condition of fullness in
Theorem~\ref{plentyN} is necessary.
\end{example}

A natural problem is to characterize the weak-$*$ closed
$\N$-bimodules in $\M$.  Given a weak-$*$ closed $\N$-bimodule
$B\subseteq \M$, one might hope to use Theorem~\ref{plentyN} to
reconstruct a given element $x\in B$ from the elements of $B\cap
\G\N(\M,\D)$.   
However, for doing this, the weak-$*$ topology is not generally the
appropriate topology.  Instead, as Mercer shows in \cite{MercerCoFoDiCrPrVNAl}, the Bures-topology turns out to be
the ``right" topology to handle such reconstruction problems.  This phenomenon was
also observed in studying bimodules in the Cartan pair case by Cameron,
Pitts and Zarikian \cite{CameronPittsZarikianBiCaMASAvNAlNoAlMeTh}
(see also \cite{DonsigFullerPittsVNAlExInSe}), and in the
crossed-product von Neumann algebras by Cameron and Smith
\cite{CameronSmithBiCrPrVNAl, CameronSmithInSuBiGeCrPrVNAl}.  Our next
goal is Theorem~\ref{BuresApx1}, which gives a method for
reconstructing $x\in B$ using $\G\N(\M,\D)\cap B$ when $B\subseteq \M$ is a
Bures-closed $\N$-bimodule.  We begin with recalling the definition of
the Bures-topology.

\begin{definition}\label{BuresTop} Let $\L\subseteq \M$ be an
inclusion of von Neumann algebras and assume there is a faithful
normal conditional expectation $E_\L:\M\rightarrow \L$.  The
\textit{$E_\L$-Bures topology} (or simply \textit{Bures topology} when
the context is clear) is the locally convex topology determined by the
family of seminorms,
\[\M\ni x \mapsto \rho(E(x^*x))^{1/2}, \rho\in \L_*^+.\]
\end{definition}

The Bures topology was introduced in \cite{BuresAbSuvNAl} in the case
when $\M$ is a factor and $\L$ is abelian.
By~\cite[Lemma~3.1]{CameronSmithBiCrPrVNAl}, for any convex set
$C\subseteq \M$, the Bures closure of $C$ contains the weak-$*$
closure of $C$, that is,
\[\text{cl}^{\text{weak-}*}(C)\subseteq \text{cl}^{\text{Bures}}(C).\]

Take $x\in \M$.  We showed in Theorem~\ref{plentyN} and
Corollary~\ref{plentyNCor} that for each $v \in \G\N(\M,\D)$,
$vE(v^*x)$ is in the $\N$-bimodule generated by $x$.  We now show
that, in the Bures topology, we can recover $x$ from the elements of
the form $vE(v^*x)$.

\begin{definition} For a \ct\ $(\M,\N,\D)$, a subset $\Y\subseteq
\G\N(\M,\D)$ is \textit{$E$-orthogonal} if whenever $v, w\in \Y$ with
$v\neq w$, $E(v^*w)=0$.
\end{definition}

\begin{theorem}\label{BuresApx1} Let $(\M,\N,\D)$ be a (not
necessarily full) \ct\ and let $\Y\subseteq \G\N(\M,\D)$ be a maximal
$E$-orthogonal subset.  Let $\Lambda$ be the set of all finite subsets
of $\Y$ directed by inclusion. For $x\in \M$ and $F\in\Lambda$, let
$x_F:=\sum_{u\in F} uE(u^*x)$.  Then the net $(x_F)_{F\in \Lambda}$
converges in the Bures topology to $x$.
\end{theorem}

\begin{proof} Let
  $\P\hookrightarrow \G\overset{q}\twoheadrightarrow \S$ be the
  extension associated to $(\M,\N,\D)$, let $\psi$ be a faithful,
  normal semi-finite weight on $\N$, let $\phi=\psi\circ E$, and let
  $(\pi_\phi,\H_\phi, \eta_\phi)$ be the semi-cyclic representation of
  $\M$ associated to $\phi$.  For any $v\in \G$, the map
  $\M\ni x\mapsto vE(v^*x)$ leaves $\fn_\phi$ invariant and depends
  only on $s=q(v)$.  Further, when $x\in\fn_\phi$,
  $\eta_\phi(x)\mapsto \eta_\phi(vE(v^*x))$ is contractive, and extends to a projection $P_s\in \B(\H_\phi)$.  (In the
  notation of Lemma~\ref{Nqcom} and Proposition~\ref{H_phi=tenpr},
  $P_s=WP_{s,\pi_\phi}W^*$).  When $s=1$, write $P$ instead of
  $P_1$.  

Arguing as
in~\cite[Lemma~2.2]{CameronPittsZarikianBiCaMASAvNAlNoAlMeTh}, we find
that the two families of semi-norms on $\M$,
\[\{\M\ni m\mapsto \sqrt{\tau(E(m^*m))}: \tau\in \N_*^+\}\dstext{and}
\{\M\ni m\mapsto \norm{\pi_\phi(m)\xi}: \xi\in \ran(P)\}\] coincide.
These families of semi-norms define the Bures topology on $\M$
(see~\cite[Definition~2.2.3]{CameronPittsZarikianBiCaMASAvNAlNoAlMeTh}).

We now argue exactly as in the proof
of~\cite[Proposition~2.4.4]{CameronPittsZarikianBiCaMASAvNAlNoAlMeTh}.
Let $n \in \fn_\phi \cap \N$.  Then
\begin{align*} \pi_\phi(x_F)\eta_\phi(n) &= \sum_{u \in F}
\eta_\phi(uE(u^*xn))\\ & = \sum_{u\in F}
\pi_\phi(u)P\pi_\phi(u)^*\eta_\phi(xn)\\ &=\sum_{u\in F}
P_{q(u)}\eta_\phi(xn)=\sum_{u\in F} P_{q(u)}\pi_\phi(x)\eta_\phi(n).
\end{align*} Hence for every $\xi \in \overline{\eta_\phi(\fn_\phi
\cap \N)}$,
\[ \pi_\phi(x_F)\xi = \sum_{u\in F} P_{q(u)}\pi_\phi(x)\xi. \] By
Proposition~\ref{PsSOT} and Proposition~\ref{H_phi=tenpr}, $I =
\sum_{u\in\Y} P_{q(u)}$ (where the sum converges strongly in
$\B(\H_\phi)$).  Thus for every $\xi \in \overline{\eta_\phi(\fn_\phi
\cap \N)}$,
\[ \pi_\phi(x_F)\xi \rightarrow \pi_\phi(x)\xi. \] Therefore, $x_F
\stackrel{\text{Bures}}{\rightarrow} x$.
\end{proof}

We now show that the Bures closure of a weak-$*$ closed $\N$-bimodule
$\B$ contains exactly the same groupoid normalizers as $\B$ itself.
The reader should note that this result gives the versions
of~\cite[Proposition~2.5.3 and
Theorem~2.5.1]{CameronPittsZarikianBiCaMASAvNAlNoAlMeTh} appropriate
to our context.

\begin{proposition}\label{samegn} Let $\B\subseteq \M$ be a weak-$*$
closed $\N$-bimodule, and set
\[\B_0=\overline{\spn}^{w*}(\G\N(\M,\D)\cap
\B)\dstext{and}\B_1:=\overline{\spn}^{\text{Bures}}(\G\N(\M,\D)\cap
\B),\] Then $\B_0$ and $\B_1$ are weak-$*$ closed $\N$-bimodules
satisfying $\B_0\subseteq \B_1$ and
\begin{equation*} \G\N(\M,\D)\cap \B_0=\G\N(\M,\D)\cap \B=
\G\N(\M,\D)\cap \B_1.
\end{equation*} Furthermore, when $(\M,\N,\D)$ is a full \ct ,
\begin{equation*} \B_0\subseteq \B\subseteq
\B_1=\overline{\B}^{\text{Bures}}.
\end{equation*}
\end{proposition}

\begin{proof} Notice that if $(x_\lambda)$ is a net in $\M$ which
Bures-converges to $x\in\M$, then for any $a\in\M$ and $b\in\N$,
$\lim^{\text{Bures}} ax_\lambda b=axb$. It follows that $\B_1$ is a 
weak-$*$ closed $\N$ bimodule.  That $\B_0$ is an $\N$-bimodule
follows from the fact that $\U(\N) \G\N(\M,\D) \U(\N)=\G\N(\M,\D)$ and
that $\spn\U(\N)=\N$.  Clearly $\B_0\subseteq \B_1$.

Suppose $v\in \G\N(\M,\D)\cap \B_1$.  If $(x_\lambda)$ is a net in
$\B$ with $\lim^{\text{Bures}}x_\lambda =v$, we find
$\lim^{\text{Bures}}v^*x_\lambda =v^*v$.  As $E$ is Bures continuous,
we have that $\lim^{\text{Bures}} E(v^*x_\lambda)=E(v^*v)=v^*v$.
Since the relative Bures topology on $\N$ is the $\sigma$-strong
topology on $\N$, $E(v^*x_\lambda)$ converges weak-$*$ to $v^*v$.  By
Corollary~\ref{plentyNCor}, $vE(v^*x_\lambda)$ is a net in $\B$
converging weak-$*$ to $v$, showing that $v\in \G\N(\M,\D)\cap \B$.
Thus, $\G\N(\M,\D)\cap \B_0=\G\N(\M,\D)\cap \B= \G\N(\M,\D)\cap \B_1.$

Now suppose $(\M,\N,\D)$ is a full \ct.  Clearly $\B_0\subseteq
\B\subseteq \B_1$.  Let
$L$ be the linear span of $\G\N(\M,\D)\cap \B$ and choose $x\in
\overline{\B}^{\text{Bures}}$.  By Theorem~\ref{plentyN}, for each
$u\in\G\N(\M,\D)$, $uE(u^*x)\in L$, so Theorem~\ref{BuresApx1} shows
$x\in \B_1$, whence $\B_1=\overline{\B}^{\text{Bures}}$.
\end{proof}

\begin{remark}{Notation} For a Bures-closed $\N$-bimodule $\B\subseteq
\M$, let $\G\N(\B,\D):=\G\N(\M,\D)\cap \B$.  Define $\Theta(\B)\subseteq
\S$ by
\begin{equation*} \Theta(\B)=q(\G\N(\B,\D)).
\end{equation*} Further, define a map $\Psi$ from the collection of
spectral sets (see Definition~\ref{def: spec set})
in $\S$ to Bures-closed $\N$-bimodules in $\M$ by
\[\Psi(A)=\overline{\spn}^{\text{Bures}}q^{-1}(A)=\overline{\spn}^{\text{Bures}}\{j(a)n:
a\in A, n\in\N\},\] which is necessarily a Bures-closed $\N$-bimodule.
\end{remark}

When $(\M,\N,\D)$ is full, Theorem~\ref{plentyN} shows that
$\G\N(\B,\D)$ is non-zero whenever $\B\neq (0)$.  We now extend the
spectral theorem for bimodules in Cartan pairs~(see
\cite[Theorem~2.5.8]{CameronPittsZarikianBiCaMASAvNAlNoAlMeTh} and
\cite[Theorem~6.3]{DonsigFullerPittsVNAlExInSe}) to the context of
Bures closed bimodules in a \ct .  Theorem~\ref{spthbi} below should
also be compared with~\cite[Theorem~4.3]{FullerPittsIsLaBuClBiCaMA}.

Suppose for $i=1,2$ that $\P_i$ are full $\N_i$-Clifford inverse monoids,
$\S$ is a Cartan inverse monoid, 
$\P_i\hookrightarrow \G_i\overset{q_i}\twoheadrightarrow
\S$ are extensions of $\S$ by $\P_i$, and let $(\M_i,\N_i,\D_i)$ be
the corresponding \ct s.  Theorem~\ref{spthbi}
implies  the  striking
fact that the lattice structure of the Bures-closed $\N_i$-bimodules
in $\M_i$ is isomorphic to the lattice of spectral sets in $\S$.
Thus, $\S$  completely determines the
lattice structure of the Bures-closed $\N_i$-bimodules 
regardless of the choice of extension of $\S$.

\begin{theorem}[Spectral Theorem for Bimodules]\label{spthbi} Let
$(\M,\N,\D)$ be a full \ct .  The map $\Theta$ is a lattice
isomorphism between the family of Bures-closed $\N$-bimodules in $\M$
and the family of spectral sets in $\S$.  Moreover,
$\Theta^{-1}=\Psi$.
\end{theorem}

\begin{proof} Let $\B$ be a Bures-closed $\N$-bimodule in $\M$ and let
$A:=\Theta(\B)$.  We will first show that $A$ is a spectral set in
$\S$.  Suppose $s\in A$ and $t\leq s$.  Then there exists an
$e\in\E(\S)$ such that $t=se$.  Write $s=q(v)$ for some $v\in
\G\N(\B,\D)$, and $e=q(p)$ for some projection $p\in\D$, we find
$t=q(vp)$, so $t\in A$.  Next, suppose that $\{s_i\}_{i\in I}$ is a
pairwise orthogonal family in $A$ and let $s=\bigvee s_i$.  For $i\neq
k$, the orthogonality of $s_i$ and $s_k$ implies that $j(s_i)$ and
$j(s_k)$ are partial isometries with orthogonal initial spaces and
orthogonal range spaces.  Therefore, the sum $\sum_{i\in I} j(s_i)$
converges strong-$*$ to an element $v\in \G\N(\M,\D)$.  As the Bures
topology is weaker than the strong-$*$ topology, $v\in \G\N(\B,\D)$.
For every $i\in I$, $q(vj(s_i^\dag s_i))= s_i$, and it follows that
$q(v)=s$.  Thus $j(s)\in \B$, and hence $s\in A$.  Therefore
$A=\Theta(\B)$ is a spectral set.

Proposition~\ref{samegn} shows that $\B$ is generated as a
$\N$-bimodule by $\B\cap \G\N(\M,\D)$.  It follows that
$\Psi(\Theta(\B))=\B$.

We now prove that $A=\Theta(\Psi(A))$.  Clearly, $A\subseteq
\Theta(\Psi(A))$.  Choose $s\in \Theta(\Psi(A))$ and let $\B:=
\Psi(A)$.  By definition, there exists $v\in \G\N(\B,\D)$ such that
$q(v)=s$.  Let
\[r=\sup\{p\in \proj(\D): q(vp)\in A\}.\] Then $q(r)$ is the maximal
idempotent in $\E(\S)$ such that $s\, q(r)\in A$.  Thus if $a\in A$,
$s\, q(r^\perp) \wedge a=0$.  Therefore, for any $n\in\N$,
\[E( (vr^\perp)^* j(a))=0=E((vr^\perp)^*j(a)n).\] Hence for any $x\in
\spn(q^{-1}(A))$, $E((vr^\perp)^*x)=0$.  As $E$ is Bures continuous,
we find that $E((vr^\perp)^*x)=0$ for every $x\in \B$.  As $vr^\perp\in
\B$ and $E$ is faithful, we obtain $vr^\perp=0$.  Hence $v=vr$.
Applying $q$ we obtain, $s=s\, q(r)\in A$, as desired.

Finally, the order preserving properties follow by the definitions of
$\Theta$ and $\Psi$.
\end{proof}

\subsection{Intermediate von Neumann algebras} Our next goal is to
give a version of Aoi's Theorem appropriate to our context.  We first
note the following technical result.

\begin{proposition}\label{modgppr} Let $\M\supseteq\N$ be an inclusion
  of von Neumann algebras and let $\D\subseteq \Z(\N)$ be a von
  Neumann subalgebra.   Assume further that there exists a faithful,
  normal conditional expectation $E:\M\rightarrow \N$.   Let $\psi$ be a faithful normal
semi-finite weight on $\N$ and let $\phi=\psi\circ E$.  Let
$\sigma_t^\phi$ be the modular automorphism group for $\phi$.  The
following statements hold.
\begin{enumerate}
\item The centralizer, $\M_\phi:=\{x\in\M: \sigma_t^\phi(x)=x \,
  \,\forall \, \, t\in\bbR\}$, for $\sigma_t^\phi$ contains $\D$.
\item If $v\in\G\N(\M,\D)$, then for every $t\in\bbR$,
$\sigma_t^\phi(v)\in\G\N(\M,\D)$.  Further $\sigma_t^\phi(v)$ is Munn
related to $v$;
\item If $\A$ is a von Neumann algebra such that $\N\subseteq\A
\subseteq \M$ and $\D$ is regular in $\A$, then there is a unique
faithful normal conditional expectation $E_\A:\M\rightarrow \A$ such
that $\phi=\phi\circ E_\A$.  In addition, $E_\A$ has the following
properties:
\begin{enumerate}
\item[(i)] $E_\A E=E E_\A=E$; and
\item[(ii)] $E_\A$ is continuous when regarded as a map of $(\M,
E\text{-Bures})$ into $(\M, E\text{-Bures})$.
\end{enumerate}
\end{enumerate}
\end{proposition}

\begin{proof} For $x\in\M$ and $d\in \D$, $E(x^*d^*dx)\leq
  \norm{d}^2E(x^*x)$ and 
\[E(d^*x^*xd)=E(x^*x)^{1/2}d^*dE(x^*x)^{1/2}\leq \norm{d}^2E(x^*x).\]
Thus $\fn_\phi$ is a $\D$-bimodule.  Recalling that
\[\fm_\phi:=\left\{\sum_{j=1}^n y_j^*x_j: n\in\bbN, x_j,
y_j\in\fn_\phi\right\},\] we see that $\fm_\phi$ is also a
$\D$-bimodule. Furthermore, for any $d\in\D$ and $x\in\fm_\phi$, we
have
\[\phi(xd)=\psi(E(xd))=\psi(E(x)d)=\psi(dE(x))=\psi(E(dx))=\phi(dx).\]
An application of~\cite[Theorem~VIII.2.6]{TakesakiThOpAlII} now gives
part (a).

Now let $v\in\G\N(\M,\D)$ and let $w=\sigma_t^\phi(v)$.  Using (a) we
have
\[w^*dw = \sigma_t^\phi(v^*dv) = v^*dv.\] Thus $w \in \G\N(\M,\D)$ and
$w$ is Munn related to $v$, proving part (b).

The regularity of $\D$ in $\A$ and part (b) show that
$\sigma_t^\phi(\A)\subseteq \A$ for every $t\in\bbR$.
Lemma~\ref{semifinite} gives $\phi|_\A$ is a faithful, semi-finite
normal weight on $\A$.  By~\cite[Theorem~IX.4.2]{TakesakiThOpAlII},
there exists a unique normal conditional expectation $E_\A:
\M\rightarrow \A$ such that $\phi\circ E_\A=\phi$.  Since $\N\subseteq
\A$, $E_\A\circ E=E$.  Let $\Phi:=E\circ E_\A$.  Then $\Phi$ is a
conditional expectation of $\M$ onto $\N$ which satisfies $\phi\circ
\Phi=\phi$.  The uniqueness assertion of
\cite[Theorem~IX.4.2]{TakesakiThOpAlII} gives $\Phi=E$.  We thus have
the formula in part (c(i)).  As $E$ is faithful, so is $E_\A$.

Finally, suppose $(x_\lambda)$ is a net in $\M$ converging to $x$ in
the $E$-Bures topology.  Applying $E$ to both sides of the inequality,
\[(E_\A(x_\lambda)-E_\A(x))^*(E_\A(x_\lambda)-E_\A(x))=E_\A(x_\lambda-x)^*E_\A(x_\lambda-x)
\leq E_\A((x_\lambda-x)^*(x_\lambda-x))\] and using the fact that $E
E_\A=E$ shows that $E_\A(x_\lambda)\rightarrow E_\A(x)$ in the
$E$-Bures topology.  Thus $E_\A$ is $E$-Bures continuous.
\end{proof}

\begin{theorem}[Aoi's Theorem for Cartan Triples]\label{Aoi} Let
$(\M,\N,\D)$ be a  \ct\ and suppose $\A$ is a von Neumann algebra such
that $\N\subseteq \A\subseteq \M$.  Then $\A$ is Bures closed.
Furthermore, if $(\M,\N,\D)$ is full, then  
$(\A,\N,\D)$ is a \ct.
\end{theorem}

\begin{proof} 
Let $\A_0$ be the weak-$*$ closure of $\spn\G\N(\A,\Z(\N))$.  Then
$\A_0$ is a von Neumann algebra and, as $\U(\N)\subseteq
\G\N(\A,\Z(\N))$, $\A_0\supseteq \N$.  Thus, $\A_0$ is a weak-$*$
closed $\N$-bimodule.  

 By Proposition~\ref{modgppr},
there exists a faithful, normal conditional expectation
$E_{\A_0}:\M\rightarrow \A_0$.   
Since
$E_{\A_0}$ is $E$-Bures continuous, it follows that $\A_0$ is $E$-Bures
closed.  By Proposition~\ref{samegn},
\[\A_0\subseteq \A\subseteq \overline{\A}^{Bures}=
  \overline{\spn}^{Bures}\G\N(\A,\Z(\N))=\overline{\A_0}^{Bures}=\A_0,\] so $\A$ is Bures
closed.

When $(\M,\N,\D)$ is full, that is, $\D=\Z(\N)$, the previous
paragraph shows that $\D$ is regular in $\A$, so $(\A,\N,\D)$ is a
\ct.
\end{proof}

\begin{remark}{Remark}  With the notation of Theorem~\ref{Aoi} and its
  proof, let $\A_{00}$ be the weak-$*$ closure of
  $\spn\G\N(\A,\D)$.  Then $\N\subseteq \A_{00}\subseteq \A_0$, and
  $\A_{00}$ is Bures closed.  However,  we have been unable to show 
$\A_{00}=\A_0$ in general, which is why  we required the fullness
hypothesis to conclude $(\A,\N,\D)$ is a \ct.    However, this
hypothesis is rather mild, and is satisfied when $\D$ is a Cartan
MASA in $\M$.  Thus Theorem~\ref{Aoi} is indeed a generalization of
Aoi's theorem for Cartan pairs.
\end{remark}

As an immediate corollary, we use the Spectral Theorem for Bimodules
to parametrize the intermediate von Neumann algebras for a full \ct .
As with Bures-closed bimodules, this parametrization depends only on
the Cartan inverse monoid and not the extension.

\begin{corollary}\label{cor: galois} Let $(\M,\N,\D)$ be a full \ct .
Set
\begin{align*}\text{vN}(\M,\N,\D)&:=\{\A: \A \text{ is a von Neumann
algebras such that }\N\subseteq \A\subseteq \M\}\text{ and}\\
\text{sub}(\S)&:=\{\T\subseteq \S: \text{$\T$ is a Cartan inverse
submonoid of $\S$ with $\E(\T)=\E(\S)$}\}.
\end{align*} Then the restriction of $\Theta$ to $\text{vN}(\M,\N,\D)$
gives a bijection between $\text{vN}(\M,\N,\D)$ and $\text{sub}(\S)$.
\end{corollary}

\section{Examples}\label{sec: examples}

In this section, we give several examples of \ct s.

\subsection{Type I examples}\label{sec: examples type I} Suppose a
Hilbert space $\H$ is decomposed as as a direct sum,
$\H=\bigoplus_{i\in I}H_i$,  where  for all $i,j \in I$,
$\dim\H_i=\dim \H_j$.  Let
$\M = \B(\H)$ and $\D$ be the von Neumann algebra generated by $\{ P_i
: P_i \hbox{ is the projection onto } \H_i , i \in I \}$.  Then $\N =
\D' = \oplus_{i \in I} \B(\H_i)$ and $(\M,\N,\D)$ is a full \ct .
Indeed,
	\begin{equation}\label{eq: type I} \M \cong \B(\ell^2(I))
\overline{\otimes} \B(\H_1), \D \cong D(\ell^2(I)) \overline{\otimes}
\bbC I_{\H_1},\ \text{and } \N \cong D(\ell^2(I)) \overline{\otimes}
\B(\H_1),
  \end{equation} where $D(\ell^2(I))$ are the diagonal operators in
$\B(\ell^2(I))$.
  
We now show every Cartan triple $(\M,\N,\D)$ with $\M=\B(\H)$ has the form
outlined above, and hence is necessarily full.  Showing that 
$\D$ is atomic is the key step.  To start, let $P$ be the
projection onto the closure of the span of the ranges of the minimal 
projections in $\D$.  We argue by contradiction to show $P=I$.  If $P \ne I$, fix a unit vector
$\eta \in P^\perp \H$ and a positive integer $n$.  Choose a maximal
chain $\fP$ in $\proj(P^\perp \D)$ (with respect to the ordering
$\leq$ in $\proj(P^\perp\D)$).  The map from $\fP$ into $[0,1]$
given by $\fP\ni R\mapsto \norm{R\eta}$
is onto
$[0,1]$ since $\fP$ has no atoms.  So for $0\leq j\leq n$, let $R_j\in\fP$
be such that $\norm{R_j\eta}=j/n$ and for $1\leq j\leq n$ put
$Q_j:=R_j-R_{j-1}$.  Thus $\norm{Q_j\eta}=1/n$ for every $j$.

If $X$
is the rank-one projection onto the span of $\eta$, then
\[ E(X) = E\left(\sum_{i=1}^n Q_i X \right) = \sum_{i=1}^n Q_i E(X) = E\left( \sum_{i=1}^n Q_i X Q_i
\right). \] As $\|Q_i X Q_i \| = 1/n$ for each $i$, $\|E(X)\| \le 1/n$ for
all choices of $n$ and so $E(X)=0$, contradicting faithfulness of $E$.
 Thus $P^\perp = 0$ and $\D$ is atomic.

Finally, since $\G\N(\M,\D)$ spans $\M$, any two atoms of $\D$, say
$A$ and $B$, must have the same dimension, since otherwise $A \M B
\cap \G\N(\M,\D) = \{0\}$, contradicting the regularity of $\D$ in
$\M$.

By the previous paragraph, for every non-zero minimal projection $P\in\Z(\M)$, $(\M P,\N P, \D P)$ is a full \ct.
As a consequence, we have the following observation.

\begin{proposition}
If $(\M,\N,\D)$ is a \ct\ with $\dim(\M)<\infty$, then $(\M,\N,\D)$ is full.
\end{proposition}

\subsection{Tensoring Cartan pairs} Equations \eqref{eq: type I}
decomposed a Cartan triple into tensor products, where $\M \cong
\B(\ell^2(I)) \overline{\otimes} \B(\H_1)$ and $\N = D(\ell^2(I))
\overline{\otimes} \B(\H_1).$ Note that $D(\ell^2(I))$ is a Cartan
subalgebra of $\B(\ell^2(I))$.  In fact, starting with any Cartan pair
we can create a Cartan triple by tensoring with a von Neumann algebra.

Suppose $\M$ is a von Neumann algebra, $\D\subseteq \M$ is a Cartan
MASA and let $\N$ be any von
Neumann algebra.  Consider the von Neumann algebras
\[ \D \otimes I_{\N} \subseteq \D \overline{\otimes} \N \subseteq \M \overline{\otimes} \N.\]
Since $\D$ is regular in $\M$ it follows that $\D \otimes I_{\N}$ is
regular in $\M \vntensor \N$.  Further, the conditional expectation $E
\colon \M \rightarrow \D$ induces a faithful conditional expectation $E\overline{\otimes}
\id_{\N} \colon \M \overline{\otimes} \N \rightarrow \D
\overline{\otimes} \N$.  By \cite[Theorem~IV.5.9 and Corollary~IV.5.10]{TakesakiThOpAlI}, $\D\vntensor \N=(\D\vntensor I_\N)^c$.   Thus $(\M \overline{\otimes} \N, \D
\overline{\otimes} \N, \D \otimes I_{\N})$ is a \ct.  Further, if $\N$
is a factor $(\M \overline{\otimes} \N, \D \overline{\otimes} \N, \D
\otimes I_{\N})$ is a full \ct.

\subsection{Crossed products by discrete groups}
\ct s arise naturally as crossed product von Neumann algebras.  In
Section~\ref{sec: abelian crossed} we will show that if $G$ is a
discrete group acting on an abelian von Neumann algebra $\D$ then $(\D
\rtimes_\alpha G, \D^c, \D)$ will always give a \ct.  If a discrete
group $G$ acts on a (not necessarily abelian) von Neumann algebra
$\N$, and $\D=\Z(\N)$, we give necessary and sufficient conditions for
$(\N\rtimes_\alpha G,\N,\D)$ to be a \ct\ in Section~\ref{sec: crossed
non-abelian}.

Let $G$ be a discrete group acting on a von Neumann algebra $\N$ by
automorphisms $\alpha$.  Let $\M= \N \rtimes_\alpha G$.  The von
Neumann algebra $\M$ is generated by a copy of $\N$ and a unitary
representation of $G$, $\{u_g\}_{g \in G}$ such that $\alpha_g(d) =
u_gdu_g^*$.  There is a faithful, normal conditional expectation
$E_{\N}$ from $\M$ onto $\N$.  Each element $x \in \M$ is uniquely
determined by a Fourier series
\[ x = \sum_{g\in G} x_g u_g, \dstext{where} x_g:=E_\N(xu_g^*) \in \N. \] This series converges
in the Bures-topology on $\M$ induced by $E_{\N}$
\cite{MercerCoFoDiCrPrVNAl}.

Cameron and Smith \cite{CameronSmithBiCrPrVNAl,
CameronSmithInSuBiGeCrPrVNAl} have studied Bures-closed bimodules and
intermediate von Neumann algebras in a large class of crossed
products.  We will see in Theorem~\ref{thm: crossed prod} that there
is overlap in our work and theirs.  However, neither work subsumes the
other.

\subsubsection{Crossed products of abelian algebras}\label{sec:
abelian crossed}
\begin{theorem}\label{thm: crossed prod is ct} Let $\D$ be an abelian
von Neumann algebra and let $G$ be a discrete group acting on $\D$ by
automorphisms $\alpha$.  Let $\M:=\D\rtimes_\alpha G$ and $\N = \D^c$.  Then $(\M,
\N, \D)$ is a \ct.
\end{theorem}

\begin{proof} Since $\D$ is clearly regular in $\D \rtimes_\alpha G$,
we only need to note that there is a faithful normal conditional
expectation from $\M$ onto $\N$.
Since there is a faithful, normal conditional expectation $E_\D$ from
$\M$ onto $\D$, $\D$ is regular in $\N$, and $\D\subseteq \N\subseteq \M$, this follows from
Proposition~\ref{modgppr}(c).
Alternatively, the existence of the conditional expectation onto $\N$ also
follows from the proof of Theorem~3.2 of
\cite{CameronSmithInSuBiGeCrPrVNAl}.  In
\cite{CameronSmithInSuBiGeCrPrVNAl} it is assumed that the action of
$G$ is by properly outer automorphisms, though this is not needed in
the proof.
\end{proof}

We give further details on the structure of this \ct.  For each
$g \in G$, let $p_g$ be the largest projection in $\D$ such that
$\alpha_g|_{\D p_g}$ is the identity.  We note $p_g$ is the Frol\'ik projection
$e_0$ for $\ad_{u_g}$ on $\D$ described in the proofs of
Lemma~\ref{Emap} and Proposition~\ref{Dqc=Nq}.
By~\cite[Lemma~2.15]{PittsStReInI}, $u_gp_g=p_gu_g\in\N$ and $E_\N(u_g)=u_gp_g$.

Since $E_\N$ is Bures continuous,
 we can explicitly describe $E_{\N}$ by
\[ E_{\N}\left(\sum_{g\in G}x_g u_g\right) = \sum_{g\in G} x_gp_g u_g. \]

\subsubsection{Crossed products of non-abelian algebras}\label{sec:
crossed non-abelian} An automorphism $\alpha$ on a von Neumann algebra
$\N$ is \emph{properly outer} if there are no nonzero central
projections $z \in \Z(\N)$ such that $\alpha|_{\N z}$ is inner.
Equivalently $\alpha$ is properly outer if and only if
$$ yx = x\alpha(y) $$
for all $y \in \N$ implies that $x=0$.  In
\cite{CameronSmithInSuBiGeCrPrVNAl} crossed products by properly outer
automorphisms are studied and the Bures-closed bimodules and
intermediate von Neumann algebras are characterized.  We show now that
the crossed products studied in \cite{CameronSmithInSuBiGeCrPrVNAl}
give rise to full \ct s under the assumption that the restriction of
the action to the center $\Z(\N)$ is also properly outer.

\begin{theorem}\label{thm: crossed prod} Let $\N$ be a von Neumann
algebra and let $G$ be a discrete group acting on $\N$ by properly
outer automorphisms $\alpha$.  Let $\M = \N \rtimes_\alpha G$.  Then
$(\M, \N, \Z(\N))$ is a Cartan triple if and only if the action of $G$
restricted to the center $\Z(\N)$ is properly outer.
\end{theorem}

\begin{proof} Suppose $x \in \Z(\N)' \cap \M$.  Let $x= \sum_{g\in
G}x_g u_g$ be the (Bures convergent) Fourier series for $x$.  Since $x\in \Z(\N)'$ it
follows that if $x_g \neq 0$ then for $d\in\Z(\N)$,
\begin{equation}\label{cpnaa1}
dx_g=E_\N(dxu_g^*)=E_\N(xu_g^*(u_gdu_g^*))=x_g\alpha_g(d).
\end{equation}

Let $J_g$ be the two-sided ideal in $\N$ generated by
$x_g$.  It follows from~\eqref{cpnaa1} that $xd=x\alpha(d)$ for all $x\in J_g$ and all $d\in
\Z(\N)$.  Since $J_g$ is a two-sided ideal, there is a central
projection $z_g\in \Z(\N)$ such that $J_g=\N z_g$.  Thus $z_gd=z_g\alpha_g(d)$
for all $d\in \Z(\N)$.  That is, $\alpha_g|_{\Z(\N)z_g}
=\mathrm{id}|_{\Z(\N)z_g}.$

It follows that  $\N=\Z(\N)^c$ if and only if for all $g\neq e$,
$\alpha_g|_{\Z(\N)}$ is properly outer.
\end{proof}

\subsection{Crossed products by equivalence relations} Igor Fulman
in~\cite{FulmanCrPrvNAlEqReThSu} studied a class of Cartan triples
which he called crossed products by an equivalence relation.  A
crossed product by an equivalence relation is a Cartan triple
satisfying the condition in Definition~\ref{defRegularizer} below,
which we also call \textit{Fulman's condition}.  
In Appendix~A we
provide a conceptual framework in terms of inverse semigroups for
Fulman's condition and show that Fulman's condition amounts to a
lifting problem.  Here we give a class of \ct s which satisfy Fulman's
condition.

Suppose $(\M,\N,\D)$ is a \ct\ with associated extension,
\[\P\hookrightarrow \G\overset{q}\twoheadrightarrow \S\] and fixed order
preserving section $j$.

\begin{definition} \label{defRegularizer} A \emph{regularizer} is a
subgroup $R\subseteq \U(\M)$ satisfying:
\begin{enumerate}
\item $\U(\N)\subseteq R\subseteq \G\N(\M,\D)$;
\item $\spn R$ is weak-$*$ dense in $\M$;
\item there is a homomorphism $\alpha\colon R \rightarrow \Aut(\N)$
such that
\begin{enumerate}
	\item if $p$ is a projection in $\D$ such that $\alpha_u|_{\D
p}=\id|_{\D p}$ then $\alpha_u|_{\N p}=\id|_{\N p}.$
	\item $\alpha_u(d) = udu^*$ for each $u\in R$ and $d\in \D$.
\end{enumerate}
\end{enumerate} We will call a map $\alpha$ satisfying conditions (i)
and (ii) of part (c) a \textit{regularizing map} for $R$.

When the \ct\ $(\M,\N,\D)$  has a regularizer, we say \textit{$(\M,\N,\D)$
satisfies Fulman's condition}.
\end{definition}

Note that if $R$ is a regularizer with regularizing map $\alpha$, then
$\ker \alpha = \U(\N)$ (\cite[Remark,
pg. 41]{FulmanCrPrvNAlEqReThSu}).

\begin{example} Let $\N$ be a von Neumann algebra and let $\D=\Z(\N)$.
Let $G$ be a discrete group acting on $\N$ by properly outer
automorphisms.  Further assume that the restriction of the action of
$G$ to $\D$ is properly outer.  Let $\M = \N \rtimes_\alpha G$.  Then
by Theorem~\ref{thm: crossed prod} $(\M,\N,\D)$ is a \ct.  Let $R$ be
the group generated by
\[ \{ u_g \colon g\in G \} \cup \{u\in \N \colon u \text{
    unitary}\}. \] Let $R:=\{u_g: g\in G\}$ and let $\alpha:
R\rightarrow \aut(\N)$ be $u_g\mapsto \ad_{u_g}$.  Since $G$ acts by properly outer automorphisms on
$\Z(\N)$, $\alpha$ is a regularizing map for $R$ so that
$R$ is a regularizer.  Thus $(\M, \N, \D)$ satisfies Fulman's
condition.
\end{example}

\appendix
\section{An Inverse Semigroup Description of Fulman's Condition}
Fulman's condition as stated in Definition~\ref{defRegularizer}, is
mysterious.  Our goal in this appendix is to establish
Theorem~\ref{FulmanRecast}, which shows Fulman's condition is
equivalent to the statement that a rather natural lifting problem
for inverse semigroups (Diagram~\ref{Fdiag}) has a positive solution.
We begin with a definition.

\begin{definition}\label{pauto} Let $(\N,\D)$ be a pair of von Neumann
algebras with $\D$ a von Neumann subalgebra of $\Z(\N)$.  A
\textit{partial automorphism of $(\N,\D)$} is a triple $(e,\alpha,f)$
consisting of projections $e, f\in \D$ and a normal $*$-isomorphism
$\alpha: f\N\rightarrow e\N$ satisfying $\alpha(e\D)=f\D$.  We will
use $\pauto(\N,\D)$ for the set of all partial automorphisms of
$(\N,\D)$.  If $f=0$ (or $e=0$) we say $(e,\alpha,f)$ is the zero
element of $\pauto(\N,\D)$.  Further, define an involution and a product in $\pauto(\N,\D)$ via,
\[(e,\alpha,f)^\dag:=(f,\alpha^{-1},e)\dstext{and}
(e_1,\alpha_1,f_1)(e_2,\alpha_2,f_2):=(\alpha_1(f_1e_2),(\alpha_1\circ\alpha_2)|_{\alpha_2^{-1}(e_2f_1)},
\alpha_2^{-1}(f_1e_2)).\] Then $\pauto(\N,\D)$ is an inverse monoid
with $0$.  Also
\[\E(\pauto(\N,\D))=\{(e, \text{id}|_{e\N},e): e\in\proj(\D)\}\] and
hence may be identified with $\proj(\D)$.  For
$\gamma=(e,\alpha,f)\in\pauto(\N,\D)$ and $x\in f\N$, we write
$\gamma(x)$ for the value of $\alpha$ at $x$.
\end{definition}

We require the following notions for an inverse semigroup $\R$.  
\begin{itemize}
\item  Two
elements $s, t\in \R$ are \textit{compatible} if $st^\dag$ and $s^\dag
t$ are idempotents; a subset $A\subseteq \R$ is \textit{compatible}
(\cite[page~26]{LawsonInSe} if every
pair of elements of $A$ is compatible.
\item $\R$ is \textit{infinitely
  distributive} (\cite[page~28]{LawsonInSe}) if whenever $I$ is an index set and $\{r_i\}_{i\in
  I}\subseteq  \R$ is such that $\bigvee_{i\in I} r_i$ exists then
for any $s\in\R$,
\[\bigvee_{i\in I} sr_i  \text{ and } \bigvee_{i\in I} r_i s \text{
    exist  and } s \left(\bigvee_{i\in I} r_i\right)=\bigvee_{i\in I}
  sr_i, \quad \left(\bigvee_{i\in I} r_i\right)s=\bigvee_{i\in I}
  r_is.\]
\item $\R$ is \textit{complete} (\cite[page~27]{LawsonInSe}) if whenever $A\subseteq \R$ is  a
  compatible subset, $\bigvee A$ exists.
\end{itemize}

\begin{lemma}\label{pautoinfdis}  The inverse semigroup
  $\pauto(\N,\D)$ is infinitely distributive and  complete.
\end{lemma}
\begin{proof}
As $\proj(\D)$ is a complete Boolean
algebra,~\cite[Proposition~1.4.20]{LawsonInSe} shows
 $\pauto(\N,\D)$ is
an infinitely distributive inverse semigroup.

We turn now to showing  $\pauto(\N,\D)$ is complete.  Given
$a=(e_a,\alpha_a,f_a)\in \pauto(\N,\D)$,  identify $a^\dag a$ with
$f_a$
and $aa^\dag$ with $e_a$, so that the source and range of $a$ belong to $\proj(\D)$. 

 First suppose
that $A\subseteq \pauto(\N,\D)$ is a finite and orthogonal set.    
Let $e=\bigvee_{a\in A} aa^\dag $ and $f=\bigvee_{a\in A} a^\dag a$.
For $n\in \N f$, $n=\sum_{a\in A} na^\dag a$ and define
\[\alpha(n):=\sum_{a\in A} \alpha_a(na).\]
Then $(e,\alpha, f)\in \pauto(\N,\D)$.  For $a\in A$, $(e,\alpha,
f)(f_a,\id|_{\N f_a}, f_a)=a$ so $a\leq (e, \alpha, f)$.   Further, if
for every $a\in A$, $a\leq (e',\alpha', f')$, then $(e,\alpha,
f)=(e',\alpha',f')(f,\id|_{\N f}, f)=(e,\alpha, f)$.  
Thus, $(e,\alpha,f)=\bigvee A$.  So joins exist for finite orthogonal sets.

\providecommand{\atom}{\operatorname{atom}}
Next, suppose $A$ is a finite compatible set.   Let $\B$
be the (finite) Boolean algebra generated by $\{a^\dag a: a\in A\}$.  The identity of $\B$ is
$f:=\bigvee_{a\in A} a^\dag a$.  Let $\atom(\B)$ be the
(finite) set of atoms of $\B$.  Let $C:=\{ap: a\in A, p\in
\atom(\B)\}$.  Then $C$ is a finite orthogonal set of elements in
$\pauto(\N,\D)$.  Let $(e,\alpha, f):=\bigvee C$.   Let $a\in A$ and
$P_a:=\{p\in \atom(\B): ap\neq 0\}$.  Then
$a=\sum_{p\in P_a} ap$  and  $a^\dag a=f_a=\sum_{p\in P_a} p\leq f$.
So for every $p\in P_a$, $(e,\alpha, f)(p,\id|_{\N p}, p)=ap$.  Thus
 \[ (e,\alpha,f)(p, \id|_{\N p},p)=ap,\dstext{whence} (e,\alpha,f)(f_a, \id_{\N f_a},f_a)=a.\]   This shows that for every $a\in
 A$, $a\leq (e,\alpha,f)$.  On the other hand, if $a\leq (e',\alpha',
 f')$ for every $a\in A$, then $(e',\alpha',f')(f,\id|_{\N f},
 f)=(e,\alpha,f)$ so $(e,\alpha,f)=\bigvee A$, showing joins exist for
 any finite compatible set.

Finally, let $A\subseteq \pauto(\N,\D)$ be an arbitrary compatible
subset.  Let $\F$ be the set of all finite subsets of $A$ ordered by
inclusion and for
$F\in\F$, let $a_F=\bigvee F$.  Notice that if $F_1\subseteq F_2$,
then $a_{F_1}\leq a_{F_2}$.    Write $a_F=(e_F, \alpha_F, f_F)$.
Let $e=\bigvee\{aa^\dag : a\in A\}=\bigvee_{F\in \F} e_F$ and
$f=\bigvee\{a^\dag a: a\in A\}=\bigvee_{F\in\F} f_F$.  For $n\in \N f$,
the net $\alpha_F(nf_F)$ converges strongly, and we define
$\alpha(n)=\lim\alpha_F(nf_F)$.    Then $(e,\alpha, f)=\bigvee A$.
\end{proof}

The inverse semigroup $\pauto(\N,\D)$ may be written as an extension,
\[\cliff{\N,\D}\hookrightarrow \pauto(\N,\D)\overset{\pi}\twoheadrightarrow
\fund{\N,\D}\] where $\cliff{\N,\D}$ is the Clifford inverse
subsemigroup of all elements of $\pauto(\N,\D)$ which are Munn related
to an idempotent, and $\fund{\N,\D}$ is the quotient of
$\pauto(\N,\D)$ by the Munn relation.

Henceforth, fix a \ct\ $(\M,\N,\D)$ with associated extension
$\P\hookrightarrow \G\overset{q}\twoheadrightarrow \S$ and
order-preserving section $j$.   We shall be 
interested in the semigroup $\pauto(\N,\D)$ arising from this \ct.

The idempotents of $\pauto(\N,\D)$ (and hence those of $\fund{\N,\D}$) may be
identified with $\E(\S)$.  
We shall show that for any Cartan triple,
there is a one-to-one inverse semigroup homomorphism $\theta:
\S\rightarrow \fund{\N,\D}$ which fixes idempotents.  Our goal in this
section is to show that Fulman's condition is satisfied if and only if
there is a lifting of $\theta$ to an inverse semigroup homomorphism
$\alpha$ so that the following diagram commutes:
\begin{equation}\label{Fdiag} \xymatrix{ & \pauto(\N,\D)\ar[d]^{\pi}\\
\S\ar[r]_-\theta\ar@{-->}[ur]^-\alpha & \fund{\N,\D}.}
\end{equation}
For $(e,\alpha,f)\in\pauto(\N,\D)$, let $[e,\alpha,f]\in
\fund{\N,\D}$ denote the Munn equivalence class of $(e,\alpha,f)$.  It
will be helpful to have an explicit description of the Munn relation
on $\pauto(\N,\D)$.

\begin{lemma}\label{Munnonpauto} For $i=1,2$, let
$(e_i,\alpha_i,f_i)\in \pauto(\N,\D)$.  The following are equivalent.
\begin{enumerate}
\item $(e_1,\alpha_1,f_1)$ is Munn related to $(e_2,\alpha_2,f_2)$;
\item for every $d\in \proj(\D)$, $\alpha_1(df_1)=\alpha_2(df_2)$;
\item $e_1=e_2$, $f_1=f_2$ and $\alpha_1|_{f_1\D}=\alpha_2|_{f_2\D}$.
\end{enumerate}
\end{lemma}

\begin{proof} Suppose (a) holds.  Then for any $d\in\proj(\D)$,
$(d,\id|_{d\N}, d)\in \E(\pauto(\N,\D))$, so
\begin{equation}\label{Munnpauto1}
(e_i,\alpha_i,f_i)(d,\id|_{d\N},d)(f_i,\alpha_i^{-1}, e_i)
=(\alpha_i(df_i), \id_{\alpha_i(df_i)\N}, \alpha_i(df_i)),
\end{equation} which yields (b).

Now suppose (b) holds.  Taking $d\in\{f_1,f_2, f_1f_2\}$ gives
$\alpha_1(f_1)=\alpha_2(f_1f_2)=\alpha_1(f_1f_2)=\alpha_2(f_2)$, so
that $f_1=f_2$ and $e_1=e_2$.  Since $f_i\D$ is generated by
$\proj(f_i\D)$, we obtain (c).

Finally, assume (c) holds.  Let $d\in \proj(\D)$.
Examining~\eqref{Munnpauto1} we obtain
\begin{align*}
(e_1,\alpha_1,f_1)(d,\id_{d\N},d)(f_1,\alpha_1^{-1},e_1)
&=(\alpha_1(df_1),\id_{\alpha_1(df_1)\N},\alpha_1(df_1))\\
&=(\alpha_2(df_2),\id_{\alpha_2(df_2)\N},\alpha_2(df_2))\\
&=(e_2,\alpha_2,f_2)(d,\id_{d\N},d)(f_2,\alpha_2^{-1},e_2).
\end{align*} Thus (a) holds and the proof is complete.
\end{proof}

We now observe that there is always a one-to-one inverse semigroup
homomorphism of $\S$ into $\fund{\N,\D}$.  Note that if $v \in \G$,
then $v$ defines a partial automorphism in $\pauto(\N,\D)$.  Indeed if
we define $\ad_v$ by
\begin{align*} \ad_v \colon & v^*v \N \rightarrow vv^* \N \\ & v^*v x
\mapsto v x v^*,
\end{align*} then $(vv^*, \ad_v, v^*v) \in \pauto(\N,\D)$.  We define
a map $\theta \colon \S \rightarrow \fund{\N,\D)}$ by
$$ \theta(s) = [ j(ss^\dag), \ad_{j(s)}, j(s^\dag s)]. $$
By Lemma~\ref{Munnonpauto}, if $v, w \in \G$ and $v$ and $w$ are Munn
equivalent, then $[vv^*, \ad_v, v^*v] = [ww^*, \ad_w, w^*w]$.  Hence
the map $\theta$ is independent of the choice of $j$.  Indeed, for any $w
\in q^{-1}\{s\}$, $\theta(s) = [ww^*, \ad_w, w^*w]$.
Thus we may use any of 
\[[j(ss^\dag), \alpha_{j(s)}, j(s^\dag s)], \quad [ss^\dag, \alpha_s,
  s^\dag s], \dstext{or} [j(ss^\dag), \alpha_s, j(s^\dag s)]\] to
denote $\theta(s)$.   
  
\begin{proposition}\label{fundtofund} The map $\theta:\S\rightarrow
\fund{\N,\D}$ given by
\[\theta(s):=[j(ss^\dag), \ad_s, j(s^\dag s)]\] is a one-to-one
homomorphism of inverse semigroups such that $\theta|_{\E(\S)}$ is an
isomorphism of $\E(\S)$ onto $\E(\fund{\N,\D})$.
\end{proposition}

\begin{proof} For $e\in\E(\S)$, $\theta(e)=[j(e), \id_{j(e)\N},
j(e)]$, so $\theta|_{\E(\S)}$ is an isomorphism of $\E(\S)$ onto
$\E(\fund{\N,\D})$.

Take $s_1, s_2\in\S$.  Then
\[ \theta(s_1s_2)=[j(s_1s_2s_2^\dag s_1^\dag), \ad_{s_1s_2},
j(s_2^\dag s_1^\dag s_1s_2)].\] On the other hand,
$\ad_{s_2}^{-1}(j(s_1^\dag s_1 s_2 s_2^\dag))=j(s_2^\dag s_1^\dag
s_1s_2)$ and $\ad_{s_1}(j(s_1^\dag s_1 s_2s_2^\dag))=j(s_1s_2s_2^\dag
s_1^\dag)$, so
\[[j(s_1s_1^\dag),\ad_{s_1}, j(s_1^\dag
s_1)][j(s_2s_2^\dag),\ad_{s_2},j(s_2^\dag s_2)] =[j(s_1s_2s_2^\dag
s_1^\dag), \ad_{s_1}\circ (\ad_{s_2}|_{j(s_2^\dag s_1^\dag
s_1s_2)\N}),j(s_2^\dag s_1^\dag s_1 s_2)].
\] Thus to show that $\theta$ is a homomorphism it suffices to show
that
$$  \ad_{s_1s_2} =  \ad_{s_1}\circ (\ad_{s_2}|_{j(s_2^\dag s_1^\dag s_1s_2)\N}). $$

Note that for each $s\in \S$ and $e\in\E(\S)$,
\begin{equation*} \theta(s)(j(e))=j(s)j(e)j(s)^*=j(ses^\dag).
\end{equation*} Hence for $e\in\E(\S)$, $\ad_{s_1s_2}(j(e)j(s_2^\dag
s_1^\dag s_1s_2))=\ad_{s_1}(\ad_{s_2}(j(e)j(s_2^\dag s_1^\dag s_1
s_2)))$.  An application of Lemma~\ref{Munnonpauto} now shows that
$\theta$ is multiplicative on $\S$.  Hence $\theta$ is an inverse
semigroup homomorphism.

If $\theta(s_1)=\theta(s_2)$, then for every $e\in\E(\S)$,
$\ad_{s_1}|_{j(e)\D}=\ad_{s_2}|_{j(e)\D}$, so that in particular,
$s_1es_1^\dag=s_2es_2^\dag$ for every $e\in\E(\S)$.  As $\S$ is
fundamental, $s_1=s_2$, whence $\theta$ is one-to-one.
\end{proof}

We now show that a regularizer may be viewed as a homomorphism of
$q(R)$ into $\aut(\N)$ satisfying Fulman's conditions.

\begin{lemma} \label{kernelFulman} Suppose a regularizer $R$ exists for $(\M,\N,\D)$.
Let $\alpha:R\rightarrow \aut(\N)$ be a regularizing map.  Then
$\alpha$ induces a one-to-one group homomorphism $\tilde\alpha \colon
q(R)\rightarrow \aut(\N,\D)$ such that for every $e\in\E(\P)$ and
$U\in R$,
\[\tilde\alpha_{q(U)}(e)=\ad_U (e)=j(q(U))ej(q(U))^*.\]
\end{lemma}

\begin{proof} By condition (c)(ii) of Definition~\ref{defRegularizer},
$(I,\alpha_U,I)\in \aut(\N,\D)$ for every $U\in R$.  Applying
condition (c)(i) of Definition~\ref{defRegularizer}, it follows that there exists a
one-to-one group homomorphism $\tilde\alpha\colon q(R)\rightarrow
\aut(\N,\D)$.  If $e\in\E(\P)$, and $U\in R$, then
$\tilde\alpha_{q(U)}(e)=\alpha_U(e)=UeU^{-1}$.
\end{proof}

\begin{lemma}
Let $R$ be a regularizer for $(\M,\N,\D)$ and let $\R:=\{q(s)e: s\in R, e\in \E(\S)\}$.
Then  $\R$ is an inverse semigroup and $\S$ is  isomorphic to the join completion of $\R$.   
\end{lemma}

\begin{proof}
A calculation shows $\R$ is an inverse semigroup, and by
definition, $\S$ is complete.   Notice
that every 
compatible order ideal of $\R$ is also a compatible order ideal of
$\S$.  Thus by the proof of~\cite[Theorem~1.4.23]{LawsonInSe}, the
join completion of $\R$ is contained in $\S$.

Take $s\in \S$ and let $t = \bigvee\{q(r) \wedge s \colon r\in R\}$.
Suppose $t\neq s$.
As $\{a \in \S \colon a \leq s\}$ is a Boolean algebra, there is a $u \in \S$ such that $u \vee t =s$ and $u \wedge t = 0$.
There is a $w\in \G\N(\M,\D)$ such that $q(w)=u$.
As $R$ densely spans $\M$, there is a $U \in R$ such that $E(U^*w)\neq 0$.
Hence $v=UE(U^*w) \neq 0$. 
Note that $v\in \G\N(\M,\D)$ and
\begin{align*}
	q(v) &= q(U)q(E(u^*w))
	= q(U)(q(U^*)q(w)\wedge 1)\\
	&= q(w) \wedge q(U) 
	\leq s \wedge q(U).
\end{align*}
Hence $q(v) \leq t$.
However, $q(v) \leq u$. Hence $u=0$, and $t=s$.
For every $r\in R$, $q(r)\wedge s\in \R$. 
Hence the completion of $\R$ is $\S$.
\end{proof}

Next we show that Fulman's condition implies that there is a
homomorphism of $\S$ into $\pauto(\N,\D)$ which lifts the map $\theta$
described in Proposition~\ref{fundtofund}.

\begin{lemma}\label{extendtoS} Suppose $\Gamma\subseteq \S$ is a group
(under the multiplication inherited from $\S$) whose unit is $1\in\S$.
Assume that $\alpha:\Gamma\rightarrow \aut(\N,\D)$ is a one-to-one
homomorphism such that for every $s\in\Gamma$, Fulman's condition (c)
is satisfied for $\alpha_s$, that is,
\begin{enumerate}
\item[(i)] if $p\in\proj(\D)$ satisfies $\alpha_s|_{p\D}=\id|_{p\D}$,
then $\alpha_s|_{p\N}=\id|_{p\N}$; and
\item[(ii)] for $d\in\D$, $\alpha_s(d)=j(s)dj(s)^*$.
\end{enumerate} Let $\S_\Gamma\subseteq \S$ be the smallest Cartan
inverse submonoid of $\S$ containing $\Gamma$ and $\E(\S)$.  Then
$\alpha$ extends uniquely to a one-to-one homomorphism
$\alpha':\S_\Gamma\rightarrow\pauto(\N,\D)$.  In addition $\pi\circ
\alpha'=\theta|_{\S_\Gamma}$.
\end{lemma}

\begin{proof} Let $\R:=\{se: s\in\Gamma, e\in\E(\S)\}$.  Since
  $\Gamma$ is a group, $\R$ is an inverse semigroup.  As in the proof
  of Lemma~\ref{kernelFulman}(b), $\S_\Gamma$ is the join completion
  of $\R$.    We shall show
  that there is a multiplicative map of $\alpha': \R\rightarrow
  \pauto(\N,\D)$. 

Suppose $s,t\in\Gamma$.  Fulman's condition (c) applied to
$s^\dag t$ shows that if $p\in\proj(\D)$ and
$\alpha_s|_{p\D}=\alpha_t|_{p\D}$, then
$\alpha_s|_{p\N}=\alpha_t|_{p\N}$.  For $s\in\Gamma$ and $e\in\E(\S)$,
define
\[\alpha'(se):=(j(ses^\dag), \alpha_s|_{j(e)\N}, j(e)).\] Note that
this is well-defined, for if $se=tf$ for some idempotents $e, f$
and $t\in\Gamma$, then $\alpha_{s^\dag t}|_{fe\D}=\id|_{fe\D}$, so
$\alpha_s|_{ef\N}=\alpha_t|_{ef\N}$.  
Thus, $\alpha' :\R\rightarrow \pauto(\N,\D)$ is well-defined.  For
$s,t\in\Gamma$ and $e,f\in\E(\S)$ a calculation shows that
$\alpha'((se)(tf))=\alpha'(st)\alpha'(tf)$, so $\alpha'$ is a
homomorphism.    Also, for any $s\in\Gamma$ and $e\in\E(\S)$,
$\pi(\alpha'(se))=\theta(se)$.  

By~\cite[Theorem~1.4.24]{LawsonInSe}, $\alpha'$ extends uniquely to a
join-preserving homomorphism of $\S_\Gamma$ into $\pauto(\N,\D)$.
Take $s \in \Gamma$. Recall $\theta(s) = [j(ss^\dag), \ad_s, j(s^\dag
s)] = [j(1), \ad_s, j(1)].$ Since $\alpha_s(d) = j(s)dj(s)^*$ for all
$d\in \D$, by Lemma~\ref{Munnonpauto}, $\pi\circ\alpha(s) =
\theta(s)$.  That $\pi\circ\alpha' = \theta|_{\Gamma_\S}$ now follows
from the definition of $\alpha'$.  Since $\theta$ is a one-to-one map
it follows that $\alpha'$ is one-to-one.
\end{proof}

We now are prepared to recast Fulman's condition as a lifting problem.

\begin{theorem}\label{FulmanRecast} Let $(\M,\N,\D)$ be a Cartan
triple with associated extension $\P\hookrightarrow
\G\overset{q}\twoheadrightarrow \S$.  Then $(\M,\N,\D)$ satisfies Fulman's
condition if and only if there exists a homomorphism of inverse
semigroups $\alpha \colon \S\rightarrow\pauto(\N,\D)$ such that
$\pi\circ\alpha=\theta$.
\end{theorem}

\begin{proof} Suppose $(\M,\N,\D)$ satisfies Fulman's condition.
Combining Lemmas~\ref{kernelFulman} and~\ref{extendtoS} we obtain a
homomorphism $\alpha:\S\rightarrow\pauto(\N,\D)$ such that
$\pi\circ\alpha=\theta$.

Conversely, suppose $\alpha:\S\rightarrow \pauto(\N,\D)$ is a
homomorphism satisfying $\pi\circ\alpha=\theta$.  Let $R:=\U(\M)\cap
\G$.  Clearly $\U(\N)\subseteq R\subseteq \G\N(\M,\D)$ and $\spn R$ is
weak-$*$ dense in $\M$.  Let $\tau:= \alpha\circ q|_R$.  Then
$\tau\colon R\rightarrow \aut(\N)$ is a homomorphism.  For $u\in R$
write $\tau_u$ instead of $\tau(u)$.

We claim that for $u\in R$ and $d\in\D$, $\tau_u(d)=udu^*$.  Since
$\pi\circ\tau=\theta$, we obtain $\pi(\tau_u)=\theta(q(u))$, that is,
$[1,\tau_u,1]=[1,\ad_{q(u)},1]$.  By Lemma~\ref{Munnonpauto}, we
obtain $\tau_u|_\D=\ad_{q(u)}|_\D$.  But, using
Proposition~\ref{fundtofund}, for every $d\in\D$,
$\ad_{q(u)}(d)=udu^*$.  The claim follows.

Suppose $e\in\proj(\D)$ and $\tau_u|_{e\D}=\id|_{e\D}$.  Let $s=q(ue)$
and note that $s^\dag s=q(e)$.  For $f\in\E(\S)$ we have
\[sfs^\dag=q(uej(f)pu^*)=q(\tau_u(ej(f)))=q(ej(f)))=q(e)fq(e)^\dag.\]
Since $\S$ is fundamental, we obtain $s=q(e)$.  So $s$ is an
idempotent.  Therefore $\alpha(s)\in \pauto(\N,\D)$ is idempotent,
which is to say that $\alpha(s)=\id|_{e\N}$.  Since
$\alpha(s)=\tau_u|_{e\N}$, we find that $\tau_u|{e\N}=\id|_{e\N}$.
This completes the proof.
\end{proof}

\begin{remark}{Remark} While we do not presently have an example, it
seems unlikely that for a general Cartan triple, this lifting problem
will have a solution.  Thus, we expect that there should be an example
of a Cartan triple which is not a crossed product by an equivalence
relation.

A sufficient condition for a solution to the lifting problem is if the
map $j \colon \S \rightarrow \G$ can be chosen so that
$j(st)^*j(s)j(t) \in \D$.  In this case the map $\alpha \colon s
\mapsto (j(ss^\dag), \ad_s, j(s^\dag s))$ can be shown to be
homomorphism.  Clearly $\theta = \pi \circ \alpha$.
\end{remark}

\providecommand{\pauto}{\operatorname{paut}}

\bibliographystyle{amsplain}

\def\cprime{$'$}
\providecommand{\bysame}{\leavevmode\hbox to3em{\hrulefill}\thinspace}
\providecommand{\MR}{\relax\ifhmode\unskip\space\fi MR }
\providecommand{\MRhref}[2]{%
  \href{http://www.ams.org/mathscinet-getitem?mr=#1}{#2}
}
\providecommand{\href}[2]{#2}

\def\cprime{$'$} \providecommand{\bysame}{\leavevmode\hbox
to3em{\hrulefill}\thinspace}
\providecommand{\MR}{\relax\ifhmode\unskip\space\fi MR }
\providecommand{\MRhref}[2]{%
\href{http://www.ams.org/mathscinet-getitem?mr=#1}{#2} }
\providecommand{\href}[2]{#2}

\end{document}